\documentclass[11pt,twoside]{amsart}

\usepackage{amssymb}
\usepackage{amsthm}
\usepackage{amsmath}
\usepackage{hyperref}
\usepackage{ifthen}
\usepackage{mathtools}

\usepackage{etoolbox}
\let\bbordermatrix\bordermatrix
\patchcmd{\bbordermatrix}{8.75}{4.75}{}{}
\patchcmd{\bbordermatrix}{\left(}{\left[}{}{}
\patchcmd{\bbordermatrix}{\right)}{\right]}{}{}

\usepackage{graphicx}
\usepackage{tikz}
\usetikzlibrary{arrows}
\usetikzlibrary{shapes,snakes}

\usepackage{ tikz }
\usetikzlibrary{ calc }

\theoremstyle{plain}
\newtheorem{theorem}{Theorem}
\newtheorem{lemma}[theorem]{Lemma}
\newtheorem{corollary}[theorem]{Corollary}
\newtheorem{proposition}[theorem]{Proposition}

\theoremstyle{definition}
\newtheorem{definition}[theorem]{Definition}
\newtheorem{example}[theorem]{Example}

\newtheorem{problem}[theorem]{Problem}

\theoremstyle{remark}

\newcommand{\set}[1]{\left\{#1\right\}}

\def\tn{\textnormal}
\def\ld{\lambda}
\def\mN{\mathbb{N}}
\def\mR{\mathbb{R}}
\def\mS{\mathbb{S}}
\def\mZ{\mathbb{Z}}

\def\A{\mathcal{A}}

\def\S{\mathcal{S}}

\def\tn{\textnormal}

\def\a{\alpha}
\def\b{\beta}

\def\ce{\coloneqq}

\newcommand{\ignore}[1]{}
\DeclareMathOperator{\SA}{SA}
\DeclareMathOperator{\LS}{LS}
\DeclareMathOperator{\BZ}{BZ}
\DeclareMathOperator{\Las}{Las}
\DeclareMathOperator{\CG}{CG}

\DeclareMathOperator{\FRAC}{FRAC}
\DeclareMathOperator{\STAB}{STAB}
\DeclareMathOperator{\LTB}{TH}

\DeclareMathOperator{\diag}{diag}
\DeclareMathOperator{\conv}{conv}
\DeclareMathOperator{\cone}{cone}
\DeclareMathOperator{\Aut}{Aut}

\title[Stable Set Polytopes with High $\LS_+$-Ranks]{\bf Stable Set Polytopes with High Lift-and-Project Ranks \\ for the Lov{\'a}sz--Schrijver SDP Operator}

\author{Yu Hin (Gary) Au}
\thanks{Yu Hin (Gary) Au: Corresponding author. Department of Mathematics and Statistics, University of Saskatchewan, Saskatoon, Saskatchewan, S7N 5E6 Canada. E-mail: gary.au@usask.ca}

\author{Levent Tun{\c c}el}
\thanks{Levent Tun{\c c}el: Research of this author was supported in part by an NSERC Discovery Grant. Department of Combinatorics and Optimization, Faculty of Mathematics, University of Waterloo, Waterloo, Ontario, N2L 3G1 Canada. E-mail: levent.tuncel@uwaterloo.ca}

\date{March 15, 2023 (revised: \today )}


\begin{document}

\maketitle 

\begin{abstract}
We study the lift-and-project rank of the stable set polytopes of graphs with respect to the Lov{\'a}sz--Schrijver SDP operator $\LS_+$. In particular, we focus on a search for relatively small graphs with high $\LS_+$-rank (i.e., the least number of iterations of the $\LS_+$ operator on the fractional stable set polytope to compute the stable set polytope). We provide families of graphs whose $\LS_+$-rank is asymptotically a linear function of its number of vertices, which is the best possible up to improvements in the constant factor. This improves upon the previous best result in this direction from 1999, which yielded graphs whose $\LS_+$-rank only grew with the square root of the number of vertices.
\end{abstract}

\section{Introduction}\label{sec1}

In combinatorial optimization, a standard approach for tackling a given problem is to encode its set of feasible solutions geometrically (e.g., via an integer programming formulation). While the exact solution set is often difficult to analyze, we can focus on relaxations of this set that have certain desirable properties (e.g., combinatorially simple to describe, approximates the underlying set of solutions well, and/or is computationally efficient to optimize over). In that regard, the \emph{lift-and-project} approach provides a systematic procedure which generates progressively tighter convex relaxations of any given $0,1$ optimization problem. In the last four decades, many procedures that fall under the lift-and-project approach have been devised (see, among others,~\cite{SheraliA90, LovaszS91, BalasCC93, Lasserre01, BienstockZ04, AuT16}), and there is an extensive body of work on their general properties and performance on a wide range of discrete optimization problems (see, for instance,~\cite{AuPhD} and the references therein). Lift-and-project operators can be classified into two groups based on the type of convex relaxations they generate: Those that generate polyhedral relaxations only (leading to Linear Programming relaxations), and those that generate spectrahedral relaxations (leading to Semidefinite Programming relaxations) which are not necessarily polyhedral. Herein, we focus on $\LS_+$ (defined in detail in Section~\ref{sec2}), the SDP-based lift-and-project operator due to Lov{\'a}sz and Schrijver~\cite{LovaszS91}, and its performance on the stable set problem of graphs. This lift-and-project operator was originally called $N_+$ in~\cite{LovaszS91}. 

$\LS_+$ is an operator on convex subsets of the hypercube $[0,1]^n$ --- given a convex set $P \subseteq [0,1]^n$, $\LS_+(P)$ is a convex set sandwiched between $P$ and the convex hull of integral points in $P$, denoted by $P_I$:
\[
P \supseteq \LS_+(P) \supseteq P_I.
\]
The convex hull of integral points in a set is also called its \emph{integer hull}. So, $P_I$ is the integer hull of $P$.

We can apply $\LS_+$ iteratively to obtain yet tighter relaxations. Given $p \in \mN$, let $\LS_+^p(P)$ be the set obtained from applying $p$ successive $\LS_+$ operations to $P$. Then it is well known (e.g., it follows from~\cite[Theorem 1.4]{LovaszS91} and the definition of $\textup{LS}_+$) that
\[
\LS_+^0(P) \ce P \supseteq \LS_+(P) \supseteq \LS_+^2(P) \supseteq \cdots \supseteq \LS_+^{n-1}(P) \supseteq  \LS_+^n(P) = P_I.
\]
Thus, $\LS_+$ generates a hierarchy of progressively tighter convex relaxations which converge to $P_I$ in no more than $n$ iterations. The reader may refer to Lov{\'a}sz and Schrijver~\cite{LovaszS91} for some other fundamental properties of the $\LS_+$ operator.

The \emph{$\LS_+$-rank} of a convex subset of the hypercube is defined to be the number of iterations it takes $\LS_+$ to return its integer hull. As we stated above, the $\LS_+$-rank of every convex set $P \subseteq [0,1]^n$ is at most $n$, and a number of elementary polytopes in $\mR^n$ have been shown to have $\LS_+$-rank $\Theta(n)$ (see, among others,~\cite{Goemans98, CookD01, GoemansT01, SchoenebeckTT07, AuT18}).

Next, given a simple, undirected graph $G=(V,E)$, we denote by $\STAB(G)$ the stable set polytope of $G$ (the convex hull of incidence vectors of stable sets in $G$), and by $\FRAC(G)$ the fractional stable set polytope of $G$ (which is defined by edge inequalities --- i.e., $x_i + x_j \leq 1$ for every edge $\set{i,j} \in E(G)$ --- and non-negativity constraints, see Section~\ref{sec22}). Then, we define the $\LS_+$-rank of a graph $G$ as the minimum non-negative integer $p$ for which $\LS_+^p(\FRAC(G)) = \STAB(G)$. We denote the $\LS_+$-rank of a graph $G$ by $r_+(G)$. When $G$ is bipartite, $\FRAC(G) = \STAB(G)$, and thus $r_+(G)=0$ in this case.

Another well-studied convex relaxation of $\STAB(G)$ is $\LTB(G)$, the Lov{\'a}sz theta body~\cite{Lovasz79} of the given graph. A remarkable property of $\LS_+$ is that $\LS_+(\FRAC(G)) \subseteq \LTB(G)$ for every graph $G$. That is, applying one iteration of $\LS_+$ to the fractional stable set polytope of a graph already yields a tractable relaxation of the stable set polytope that is stronger than $\LTB(G)$ (see~\cite[Lemma 2.17 and Corollary 2.18]{LovaszS91} for a proof). Therefore, it follows that $r_+(G) \leq 1$ when $G$ is a perfect graph (defined in Section~\ref{sec22}), as it is well-known that $\LTB(G) = \STAB(G)$ in this case~\cite{Lovasz79}.  For more analyses of various lift-and-project relaxations of the stable set problem, see (among others)~\cite{LovaszS91, deKlerkP2002, Laurent03, LiptakT03, BienstockO04, EscalanteMN06, GvozdenovicL07, PenaVZ2007, GiandomenicoLRS09, GiandomenicoRS13, AuT16, AuLT23}.

On the other hand, while the stable set problem is known to be strongly $\mathcal{NP}$-hard, hardness results for $\LS_+$ (or any other lift-and-project operator utilizing semidefinite programming) on the stable set problem have been relatively scarce. Prior to this manuscript, the worst (in terms of performance by $\LS_+$) family of graphs was given by the line graphs of odd cliques. Using the fact that the $\LS_+$-rank of the fractional matching polytope of the $(2k+1)$-clique is $k$~\cite{StephenT99}, the natural correspondence between the matchings in a given graph and the stable sets in the corresponding line graph, as well as the properties of the $\LS_+$ operator, it follows that the fractional stable set polytope of the line graph of the $(2k+1)$-clique (which contains $\binom{2k+1}{2} = k(2k+1)$ vertices) has $\LS_+$-rank $k$, giving a family of graphs $G$ with $r_+(G) = \Theta\left(\sqrt{|V(G)|}\right)$. This lower bound on the $\LS_+$-rank of the fractional stable set polytopes has not been improved since 1999. In this manuscript, we present what we believe is the first known family of graphs whose $\LS_+$-rank is asymptotically a linear function of the number of vertices. 

\subsection{A family of challenging graphs $\left\{H_k\right\}$ for the $\LS_+$ operator} For a positive integer $k$, let $[k] \ce \set{1,2, \ldots, k}$.  We first define the family of graphs that pertains to our main result.

\begin{definition}\label{defnHk}
Given an integer $k \geq 2$, define $H_k$ to be the graph where 
\[
V(H_k) \ce \set{ i_p : i \in [k], p \in \set{0,1,2}},
\]
and the edges of $H_k$ are
\begin{itemize}
\item
$\set{i_0, i_1}$ and $\set{i_1,i_2}$ for every $i \in [k]$;
\item
$\set{i_0, j_2}$ for all $i,j \in [k]$ where $i \neq j$.
\end{itemize}
\end{definition}

\def\y{0.70}
\def\sc{2}

\begin{figure}[ht!]
\begin{center}
\begin{tabular}{ccc}

\def\x{270 - 180/3}
\def\z{360/3}

\begin{tikzpicture}[scale=\sc, thick,main node/.style={circle, minimum size=4mm, inner sep=0.1mm,draw,font=\tiny\sffamily}]
\node[main node] at ({ \y* cos(\x + (0)*\z) + (1-\y)*cos(\x+(-1)*\z)},{ \y* sin(\x+(0)*\z) + (1-\y)*sin(\x+(-1)*\z)}) (1) {$1_0$};
\node[main node] at ({cos(\x+(0)*\z)},{sin(\x+(0)*\z)}) (2) {$1_1$};
\node[main node] at ({ \y* cos(\x+(0)*\z) + (1-\y)*cos(\x+(1)*\z)},{ \y* sin(\x+(0)*\z) + (1-\y)*sin(\x+(1)*\z)}) (3) {$1_2$};

\node[main node] at ({ \y* cos(\x+(1)*\z) + (1-\y)*cos(\x+(0)*\z)},{ \y* sin(\x+(1)*\z) + (1-\y)*sin(\x+(0)*\z)}) (4) {$2_0$};
\node[main node] at ({cos(\x+(1)*\z)},{sin(\x+(1)*\z)}) (5) {$2_1$};
\node[main node] at ({ \y* cos(\x+(1)*\z) + (1-\y)*cos(\x+(2)*\z)},{ \y* sin(\x+(1)*\z) + (1-\y)*sin(\x+(2)*\z)}) (6) {$2_2$};

\node[main node] at ({ \y* cos(\x+(2)*\z) + (1-\y)*cos(\x+(1)*\z)},{ \y* sin(\x+(2)*\z) + (1-\y)*sin(\x+(1)*\z)}) (7) {$3_0$};
\node[main node] at ({cos(\x+(2)*\z)},{sin(\x+(2)*\z)}) (8) {$3_1$};
\node[main node] at ({ \y* cos(\x+(2)*\z) + (1-\y)*cos(\x+(3)*\z)},{ \y* sin(\x+(2)*\z) + (1-\y)*sin(\x+(3)*\z)}) (9) {$3_2$};

 \path[every node/.style={font=\sffamily}]
(2) edge (1)
(2) edge (3)
(5) edge (4)
(5) edge (6)
(8) edge (7)
(8) edge (9)
(1) edge (6)
(1) edge (9)
(4) edge (3)
(4) edge (9)
(7) edge (3)
(7) edge (6);
\end{tikzpicture}

&

\def\x{270 - 180/4}
\def\z{360/4}

\begin{tikzpicture}[scale=\sc, thick,main node/.style={circle, minimum size=4mm, inner sep=0.1mm,draw,font=\tiny\sffamily}]

\node[main node] at ({ \y* cos(\x+(0)*\z) + (1-\y)*cos(\x+(-1)*\z)},{ \y* sin(\x+(0)*\z) + (1-\y)*sin(\x+(-1)*\z)}) (1) {$1_0$};
\node[main node] at ({cos(\x+(0)*\z)},{sin(\x+(0)*\z)}) (2) {$1_1$};
\node[main node] at ({ \y* cos(\x+(0)*\z) + (1-\y)*cos(\x+(1)*\z)},{ \y* sin(\x+(0)*\z) + (1-\y)*sin(\x+(1)*\z)}) (3) {$1_2$};

\node[main node] at ({ \y* cos(\x+(1)*\z) + (1-\y)*cos(\x+(0)*\z)},{ \y* sin(\x+(1)*\z) + (1-\y)*sin(\x+(0)*\z)}) (4) {$2_0$};
\node[main node] at ({cos(\x+(1)*\z)},{sin(\x+(1)*\z)}) (5) {$2_1$};
\node[main node] at ({ \y* cos(\x+(1)*\z) + (1-\y)*cos(\x+(2)*\z)},{ \y* sin(\x+(1)*\z) + (1-\y)*sin(\x+(2)*\z)}) (6) {$2_2$};

\node[main node] at ({ \y* cos(\x+(2)*\z) + (1-\y)*cos(\x+(1)*\z)},{ \y* sin(\x+(2)*\z) + (1-\y)*sin(\x+(1)*\z)}) (7) {$3_0$};
\node[main node] at ({cos(\x+(2)*\z)},{sin(\x+(2)*\z)}) (8) {$3_1$};
\node[main node] at ({ \y* cos(\x+(2)*\z) + (1-\y)*cos(\x+(3)*\z)},{ \y* sin(\x+(2)*\z) + (1-\y)*sin(\x+(3)*\z)}) (9) {$3_2$};

\node[main node] at ({ \y* cos(\x+(3)*\z) + (1-\y)*cos(\x+(2)*\z)},{ \y* sin(\x+(3)*\z) + (1-\y)*sin(\x+(2)*\z)}) (10) {$4_0$};
\node[main node] at ({cos(\x+(3)*\z)},{sin(\x+(3)*\z)}) (11) {$4_1$};
\node[main node] at ({ \y* cos(\x+(3)*\z) + (1-\y)*cos(\x+(4)*\z)},{ \y* sin(\x+(3)*\z) + (1-\y)*sin(\x+(4)*\z)}) (12) {$4_2$};

 \path[every node/.style={font=\sffamily}]
(2) edge (1)
(2) edge (3)
(5) edge (4)
(5) edge (6)
(8) edge (7)
(8) edge (9)
(11) edge (10)
(11) edge (12)
(1) edge (6)
(1) edge (9)
(1) edge (12)
(4) edge (3)
(4) edge (9)
(4) edge (12)
(7) edge (3)
(7) edge (6)
(7) edge (12)
(10) edge (3)
(10) edge (6)
(10) edge (9);
\end{tikzpicture}

&

\def\x{270 - 180/5}
\def\z{360/5}

\begin{tikzpicture}[scale=\sc, thick,main node/.style={circle, minimum size=4mm, inner sep=0.1mm,draw,font=\tiny\sffamily}]
\node[main node] at ({ \y* cos(\x+(0)*\z) + (1-\y)*cos(\x+(-1)*\z)},{ \y* sin(\x+(0)*\z) + (1-\y)*sin(\x+(-1)*\z)}) (1) {$1_0$};
\node[main node] at ({cos(\x+(0)*\z)},{sin(\x+(0)*\z)}) (2) {$1_1$};
\node[main node] at ({ \y* cos(\x+(0)*\z) + (1-\y)*cos(\x+(1)*\z)},{ \y* sin(\x+(0)*\z) + (1-\y)*sin(\x+(1)*\z)}) (3) {$1_2$};

\node[main node] at ({ \y* cos(\x+(1)*\z) + (1-\y)*cos(\x+(0)*\z)},{ \y* sin(\x+(1)*\z) + (1-\y)*sin(\x+(0)*\z)}) (4) {$2_0$};
\node[main node] at ({cos(\x+(1)*\z)},{sin(\x+(1)*\z)}) (5) {$2_1$};
\node[main node] at ({ \y* cos(\x+(1)*\z) + (1-\y)*cos(\x+(2)*\z)},{ \y* sin(\x+(1)*\z) + (1-\y)*sin(\x+(2)*\z)}) (6) {$2_2$};

\node[main node] at ({ \y* cos(\x+(2)*\z) + (1-\y)*cos(\x+(1)*\z)},{ \y* sin(\x+(2)*\z) + (1-\y)*sin(\x+(1)*\z)}) (7) {$3_0$};
\node[main node] at ({cos(\x+(2)*\z)},{sin(\x+(2)*\z)}) (8) {$3_1$};
\node[main node] at ({ \y* cos(\x+(2)*\z) + (1-\y)*cos(\x+(3)*\z)},{ \y* sin(\x+(2)*\z) + (1-\y)*sin(\x+(3)*\z)}) (9) {$3_2$};

\node[main node] at ({ \y* cos(\x+(3)*\z) + (1-\y)*cos(\x+(2)*\z)},{ \y* sin(\x+(3)*\z) + (1-\y)*sin(\x+(2)*\z)}) (10) {$4_0$};
\node[main node] at ({cos(\x+(3)*\z)},{sin(\x+(3)*\z)}) (11) {$4_1$};
\node[main node] at ({ \y* cos(\x+(3)*\z) + (1-\y)*cos(\x+(4)*\z)},{ \y* sin(\x+(3)*\z) + (1-\y)*sin(\x+(4)*\z)}) (12) {$4_2$};

\node[main node] at ({ \y* cos(\x+(4)*\z) + (1-\y)*cos(\x+(3)*\z)},{ \y* sin(\x+(4)*\z) + (1-\y)*sin(\x+(3)*\z)}) (13) {$5_0$};
\node[main node] at ({cos(\x+(4)*\z)},{sin(\x+(4)*\z)}) (14) {$5_1$};
\node[main node] at ({ \y* cos(\x+(4)*\z) + (1-\y)*cos(\x+(5)*\z)},{ \y* sin(\x+(4)*\z) + (1-\y)*sin(\x+(5)*\z)}) (15) {$5_2$};

 \path[every node/.style={font=\sffamily}]
(2) edge (1)
(2) edge (3)
(5) edge (4)
(5) edge (6)
(8) edge (7)
(8) edge (9)
(11) edge (10)
(11) edge (12)
(14) edge (13)
(14) edge (15)
(1) edge (6)
(1) edge (9)
(1) edge (12)
(1) edge (15)
(4) edge (3)
(4) edge (9)
(4) edge (12)
(4) edge (15)
(7) edge (3)
(7) edge (6)
(7) edge (12)
(7) edge (15)
(10) edge (3)
(10) edge (6)
(10) edge (9)
(10) edge (15)
(13) edge (3)
(13) edge (6)
(13) edge (9)
(13) edge (12);
\end{tikzpicture}
\\
$H_3$ & $H_4$ & $H_5$ 
\end{tabular}
\caption{Several graphs in the family $H_k$}\label{figH_k}
\end{center}
\end{figure}
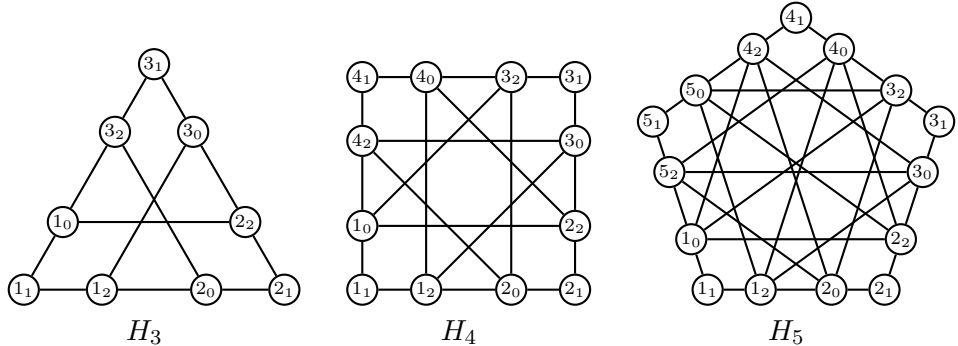

Figure~\ref{figH_k} illustrates the graphs $H_k$ for several small values of $k$. Note that $H_2$ is the
cycle on six vertices. Moreover, $H_k$ can be obtained from a complete bipartite graph by fixing a perfect
matching and replacing each edge by a path of length two (subdividing each matching edge).
Figure~\ref{figH_k2} represents a drawing emphasizing this second feature.

The following is the main result of this paper. Its proof is given in Section~\ref{sec4}.

\begin{theorem}\label{thmHk}
For every $k \geq 3$, the $\LS_+$-rank of the fractional stable set polytope of $H_k$ is at least $\frac{1}{16}|V(H_k)|$.
\end{theorem}

We remark that, given $p \in \mN$ and a polytope $P \subseteq [0,1]^n$ with $\Omega(n^c)$ facets for some constant $c$, the straightforward formulation of $\LS_+^p(P)$ is an SDP with size $n^{\Omega(p)}$. Since the fractional stable set polytope of the graph $H_k$ has dimension $n=3k$ with $\Omega(n^2)$ facets, Theorem~\ref{thmHk} implies that the SDP described by $\LS_+$ that fails to exactly represent the stable set polytope of $H_k$
 has size $n^{\Omega(n)}$. Thus, the consequence of Theorem~\ref{thmHk} on the size of the SDPs generated from the $\LS_+$ operator is incomparable with the extension complexity bound due to Lee, Raghavendra, and Steurer~\cite{LeeRS15}, who showed (in the context of SDP extension complexity) that it takes an SDP of size  $2^{\Omega(n^{1/13})}$ to exactly represent the stable set polytope of a general $n$-vertex graph as a projection of a spectrahedron (such sets are also called \emph{spectrahedral shadows}). That is, while the general lower bound by ~\cite{LeeRS15} covers every lifted-SDP formulation, our lower bound is significantly larger but it only applies to a specific family of lifted-SDP formulations.

\subsection{Organization of the paper}\label{sec11}
In Section~\ref{sec2}, we introduce the $\LS_+$ operator and the stable set problem, and establish some notations and basic facts that will aid our subsequent discussion. In Section~\ref{sec3}, we study the family of graphs $H_k$, their stable set polytope and set up some fundamental facts and the proof strategy. We then prove Theorem~\ref{thmHk} in Section~\ref{sec4}. In Section~\ref{sec5} we determine the Chv{\'a}tal--Gomory rank of the stable set polytope of the graphs $H_k$. In Section~\ref{sec6}, we show that the results in Sections~\ref{sec3} and~\ref{sec4} readily lead to the discovery of families of vertex-transitive graphs whose $\LS_+$-rank also exhibits asymptotically linear growth. Finally, in Section~\ref{sec7}, we close by mentioning some natural research directions inspired by these new findings.

\section{Preliminaries}\label{sec2}

In this section, we establish the necessary definitions and notation for our subsequent analysis. 

\subsection{The lift-and-project operator $\LS_+$}\label{sec21}

Here, we define the lift-and-project operator $\LS_+$ due to Lov{\'a}sz and Schrijver~\cite{LovaszS91} and mention some of its basic properties. Given a convex set $P \subseteq [0,1]^n$, we define the cone
\[
\cone(P) \ce \set{ \begin{bmatrix} \ld \\ \ld x \end{bmatrix} : \ld \geq 0, x \in P},
\]
and index the new coordinate by $0$. Given a vector $x$ and an index $i$, we may refer to the $i$-entry in $x$ by $x_i$ or $[x]_i$. All vectors are column vectors, so here the transpose of $x$, $x^{\top}$, is a row vector. Next, given a symmetric matrix $M \in \mR^{n \times n}$ (which necessarily has only real eigenvalues), we say that $M$ is \emph{positive semidefinite} and write $M \succeq 0$ if all eigenvalues of $M$ are non-negative. We let $\mathbb{S}_+^n$ denote the set of $n$-by-$n$ symmetric positive semidefinite matrices, and $\diag(Y)$ be the vector formed by the diagonal entries of a square matrix $Y$. We also let $e_i$ be the $i^{\tn{th}}$ unit vector.

Given $P \subseteq [0,1]^n$, the operator $\LS_+$ first \emph{lifts} $P$ to the following set of matrices:
\[
\widehat{\LS}_+(P) \ce \set{ Y \in \mS_+^{n+1} : Ye_0 = \diag(Y), Ye_i, Y(e_0-e_i) \in \cone(P)~\forall i \in [n] }.
\]
It then \emph{projects} the set back down to the following set in $\mR^n$:
\[
\LS_+(P) \ce \set{ x \in \mR^n : \exists Y \in \widehat{\LS}_+(P), Ye_0 = \begin{bmatrix} 1 \\ x \end{bmatrix}}.
\]
Given $x \in \LS_+(P)$, we say that $Y \in \widehat{\LS}_+(P)$ is a \emph{certificate matrix} for $x$ if $Ye_0 = \begin{bmatrix} 1 \\ x \end{bmatrix}$. The following is a foundational property of $\LS_+$~\cite{LovaszS91}.

\begin{lemma}\label{lem2.0}
Let $P \subseteq [0,1]^n$ be a convex set. Then,
\[
P \supseteq \LS_+(P) \supseteq P_I.
\]
\end{lemma}

\begin{proof}
Let $x \in \LS_+(P)$, and let $Y \in \widehat{\LS}_+(P)$ be a certificate matrix for $x$. Since $Ye_0 = Ye_i + Y(e_0 - e_i)$ for any index $i \in [n]$ and $\widehat{\LS}_+$ imposes that $Ye_i, Y(e_0 - e_i) \in \cone(P)$, it follows that $Ye_0 \in \cone(P)$, and thus $x \in P$. On the other hand, given any integral vector $x \in P \cap \set{0,1}^n$, observe that $Y \ce \begin{bmatrix} 1 \\ x \end{bmatrix}\begin{bmatrix} 1 \\ x \end{bmatrix}^{\top} \in \widehat{\LS}_+(P)$, and so $x \in \LS_+(P)$. Thus, since $P_I \ce \conv\left( P \cap \set{0,1}^n \right)$ (i.e., the integer hull of $P$), we deduce that
\[
P_I \subseteq \LS_+(P) \subseteq P
\]
holds. 
\end{proof}

Therefore, $\LS_+(P)$ contains the same set of integral solutions as $P$. Moreover, if $P$ is a tractable set (i.e., one can optimize a linear function over $P$ in polynomial time), then so is $\LS_+(P)$. It is also known that $\LS_+(P)$ is strictly contained in $P$ unless $P = P_I$. Thus, while it is generally $\mathcal{NP}$-hard to optimize over the integer hull $P_I$, $\LS_+(P)$ offers a tractable relaxation of $P_I$ that is tighter than the initial relaxation $P$. Again, the reader may refer to Lov{\'a}sz and Schrijver~\cite{LovaszS91} additional discussion and properties of the $\LS_+$ operator.

\subsection{The stable set polytope and the $\LS_+$-rank of graphs}\label{sec22}

Given a simple, undirected graph $G \ce (V(G), E(G))$, we define its \emph{fractional stable set polytope} to be
\[
\FRAC(G) \ce \set{ x \in [0,1]^{V(G)} : x_i + x_j \leq 1, \forall \set{i, j } \in E(G)}.
\]
We also define 
\[
\STAB(G) \ce \FRAC(G)_I = \conv\left( \FRAC(G) \cap \set{0,1}^{V(G)} \right)
\]
to be the \emph{stable set polytope} of $G$. Notice that $\STAB(G)$ is exactly the convex hull of the incidence vectors of stable sets in $G$. Also, to reduce cluttering, we will write $\LS_+^p(G)$ instead of $\LS_+^p(\FRAC(G))$. 

Given a graph $G$, recall that we let $r_+(G)$ denote the \emph{$\LS_+$-rank} of $G$, which is defined to be the smallest integer $p$ where $\LS_+^p(G) = \STAB(G)$. More generally, given a linear inequality $a^{\top} x \leq \beta$ valid for $\STAB(G)$, we define its \emph{$\LS_+$-rank} to be the smallest integer $p$ for which $a^{\top} x \leq \beta$ is a valid inequality of $\LS_+^p(G)$. Then $r_+(G)$ can be alternatively defined as the maximum $\LS_+$-rank over all valid inequalities of $\STAB(G)$.

It is well known that $r_+(G) = 0$ (i.e., $\STAB(G) = \FRAC(G)$) if and only if $G$ is bipartite. Next, given a graph $G$ and $C \subseteq V(G)$, we say that $C$ is a \emph{clique} if every pair of vertices in $C$ is joined by edge in $G$. Then observe that the \emph{clique inequality} $\sum_{i \in C} x_i \leq 1$ is valid for $\STAB(G)$. A graph $G$ is \emph{perfect} if $\STAB(G)$ is defined by only clique and non-negativity inequalities. Since clique inequalities of graphs have $\LS_+$-rank $1$ (see~\cite[Lemma 1.5]{LovaszS91} for a proof), we see that $r_+(G) \leq 1$ for all perfect graphs $G$. 

The graphs $G$ where $r_+(G) \leq 1$ are commonly called \emph{$\LS_+$-perfect graphs}. In addition to perfect graphs, it is known that odd holes, odd antiholes, and odd wheels (among others) are also $\LS_+$-perfect. While it remains an open problem to find a simple combinatorial characterization of $\LS_+$-perfect graphs, there has been significant recent interest and progress on this front~\cite{BianchiENT13, BianchiENT14, BianchiENT17, Wagler22, BianchiENW23}.

Next, we mention two simple graph operations that have been critical to the analyses of the $\LS_+$-ranks of graphs. Given a graph $G$ and $S \subseteq V(G)$, we let $G-S$ denote the subgraph of $G$ induced by the vertices in $V(G) \setminus S$, and call $G-S$ the graph obtained by the \emph{deletion} of $S$. (When $S = \set{i}$ for some vertex $i$, we simply write $G-i$ instead of $G - \set{i}$.) Next, given $i \in V(G)$, let $\Gamma(i) \ce \set{ j \in V(G) : \set{i,j} \in E(G)}$ (i.e., $\Gamma(i)$ is the set of vertices that are adjacent to $i$). Then the graph obtained from the \emph{destruction} of $i$ in $G$ is defined as
\[
G \ominus i \ce G - ( \set{i} \cup \Gamma(i)).
\]
Then we have the following.
\begin{theorem}\label{thmDeleteDestroy}
For every graph $G$,
\begin{itemize}
\item[(i)]
\cite[Corollary 2.16]{LovaszS91} $r_+(G) \leq \max \set{ r_+(G \ominus i) : i \in V(G) } + 1$;
\item[(ii)]
\cite[Theorem 36]{LiptakT03} $r_+(G) \leq \min \set{ r_+(G - i) : i \in V(G) } + 1$.
\end{itemize}
\end{theorem}

We next mention a fact that is folklore, with related insights dating back to Balas' work on disjunctive programming in the 1970s (see, for instance,~\cite{Balas98} and \cite[Lemma 3.2]{GoemansT01}). 

\begin{lemma}\label{lemLS+F}
Let $F$ be a face of $[0,1]^n$, and let $P \subseteq [0,1]^n$ be a convex set. Then
\[
\LS_+^p(P \cap F) = \LS_+^p(P) \cap F
\]
for every integer $p \geq 1$.
\end{lemma}

\begin{proof}
It suffices to prove the claim for $p=1$, as the general result would follow from iterative application of this case. First, let $x \in \LS_+(P \cap F)$. Since $\LS_+(P \cap F) \subseteq P \cap F \subseteq F$, it follows that $x \in F$. Now let $Y \in \widehat{\LS}_+(P \cap F)$ be a certificate matrix for $x$. Then $Ye_i, Y(e_0 - e_i) \in \cone(P \cap F) \subseteq \cone(P)$, and so it follows that $Y \in \widehat{\LS}_+(P)$, which certifies that $x \in \LS_+(P)$.

Conversely, let $x \in \LS_+(P) \cap F$, and let $Y \in \widehat{\LS}_+(P)$ be a certificate matrix for $x$. Then $Ye_i, Y(e_0 - e_i) \in \cone(P)$ for every $i \in [n]$. Now since $x \in F$ and $\begin{bmatrix} 1 \\ x \end{bmatrix} = Ye_i + Y(e_0-e_i)$, it follows (from $F$ being a face of $[0,1]^n$ and that $Ye_i, Y(e_0-e_i) \in \cone(P)$ and $\cone(P)$ is a subset of the cone of the hypercube $[0,1]^n$) that $Ye_i, Y(e_0 - e_i) \in \cone(F)$ for every $i \in [n]$. This implies that $Ye_i, Y(e_0-e_i) \in \cone(P \cap F)$, which in turn implies that $x \in \LS_+(P \cap F)$.
\end{proof}

Using Lemma~\ref{lemLS+F}, we obtain another elementary property of $\LS_+$ that will be useful.
 
 \begin{lemma}\label{lemInducedSubgraph}
 Let $G$ be a graph, and let $G'$ be an induced subgraph of $G$. Then $r_+(G') \leq r_+(G)$.
 \end{lemma}
 
 \begin{proof}
First, there is nothing to prove if $r_+(G') = 0$, so we assume that $r_+(G') = p+1$ for some $p \geq 0$. Thus, there exists $\bar{x}' \in \LS_+^p(G') \setminus \STAB(G')$.

Now define $\bar{x} \in \mR^{V(G)}$ where $\bar{x}_i = \bar{x}_i'$ if $i \in V(G')$ and $\bar{x}_i = 0$ otherwise.  It is easy to see that $\bar{x}' \not\in \STAB(G') \Rightarrow \bar{x} \not\in \STAB(G)$. Then, applying Lemma~\ref{lemLS+F} with $P \ce \FRAC(G)$ and $F \ce \set{ x \in [0,1]^{V(G)} : x_i = 0~\forall i \in V(G) \setminus V(G')}$, we obtain that $\bar{x} \in \LS_+^p(G)$. Thus, we conclude that $\bar{x} \in \LS_+^{p}(G) \setminus \STAB(G)$, and $r_+(G) \geq p+1$.
 \end{proof}
 
We are interested in studying relatively small graphs with high $\LS_+$-rank --- that is, graphs whose stable set polytope is difficult to obtain for $\LS_+$. First, Lipt{\'a}k and the second author~\cite[Theorem 39]{LiptakT03} proved the following general upper bound:

\begin{theorem}\label{thmNover3}
For every graph $G$, $r_+(G) \leq \left\lfloor \frac{ |V(G)|}{3} \right\rfloor$.
\end{theorem}

In Section~\ref{sec4}, we prove that the family of graphs $H_k$ satisfies $r_+(H_k) = \Theta(|V(H_k)|)$. This shows that Theorem~\ref{thmNover3} is asymptotically tight, and rules out the possibility of a sublinear upper bound on the $\LS_+$-rank of a general graph.

\section{Analyzing $H_k$ and exploiting symmetries}\label{sec3}

\subsection{The graphs $H_k$ and their basic properties}\label{sec31}

Recall the family of graphs $H_k$ defined in Section~\ref{sec1} (Definition~\ref{defnHk}). For convenience, we let $[k]_p \ce \set{ j_p : j \in [k]}$ for each $p \in \set{0,1,2}$. Then, as mentioned earlier, one can construct $H_k$ by starting with a complete bipartite graph with bipartitions $[k]_0$ and $[k]_2$, and then for every $j \in [k]$ subdividing the edge $\set{j_0, j_2}$ into a path of length $2$ and labelling the new vertex $j_1$. Figure~\ref{figH_k2} illustrates alternative drawings for $H_k$ which highlight this aspect of the family of graphs.

\def\z{0.25}
\begin{figure}[ht!]
\begin{center}
\begin{tabular}{ccc}

\begin{tikzpicture}[scale=\sc, thick,main node/.style={circle, minimum size=4mm, inner sep=0.1mm,draw,font=\tiny\sffamily}]

\node[main node] at (-1,0) (1) {$1_0$};
\node[main node] at (0, {3*\z}) (2) {$1_1$};
\node[main node] at (1,0) (3) {$1_2$};

\node[main node] at (-1,{1*\z} ) (4) {$2_0$};
\node[main node] at (0, {4*\z}) (5) {$2_1$};
\node[main node] at (1,{1*\z}) (6) {$2_2$};

\node[main node] at (-1,{2*\z}) (7) {$3_0$};
\node[main node] at (0, {5*\z}) (8) {$3_1$};
\node[main node] at (1,{2*\z}) (9) {$3_2$};

 \path[every node/.style={font=\sffamily}]
(2) edge (1)
(2) edge (3)
(5) edge (4)
(5) edge (6)
(8) edge (7)
(8) edge (9)
(1) edge (6)
(1) edge (9)
(4) edge (3)
(4) edge (9)
(7) edge (3)
(7) edge (6);
\end{tikzpicture}

&

\begin{tikzpicture}[scale=\sc, thick,main node/.style={circle, minimum size=4mm, inner sep=0.1mm,draw,font=\tiny\sffamily}]

\node[main node] at (-1,0) (1) {$1_0$};
\node[main node] at (0, {4*\z}) (2) {$1_1$};
\node[main node] at (1,0) (3) {$1_2$};

\node[main node] at (-1,{1*\z} ) (4) {$2_0$};
\node[main node] at (0, {5*\z}) (5) {$2_1$};
\node[main node] at (1,{1*\z}) (6) {$2_2$};

\node[main node] at (-1,{2*\z}) (7) {$3_0$};
\node[main node] at (0, {6*\z}) (8) {$3_1$};
\node[main node] at (1,{2*\z}) (9) {$3_2$};

\node[main node] at (-1,{3*\z}) (10) {$4_0$};
\node[main node] at (0, {7*\z}) (11) {$4_1$};
\node[main node] at (1,{3*\z}) (12) {$4_2$};

 \path[every node/.style={font=\sffamily}]
(2) edge (1)
(2) edge (3)
(5) edge (4)
(5) edge (6)
(8) edge (7)
(8) edge (9)
(11) edge (10)
(11) edge (12)
(1) edge (6)
(1) edge (9)
(1) edge (12)
(4) edge (3)
(4) edge (9)
(4) edge (12)
(7) edge (3)
(7) edge (6)
(7) edge (12)
(10) edge (3)
(10) edge (6)
(10) edge (9);
\end{tikzpicture}

&

\def\x{360/5}

\begin{tikzpicture}[scale=\sc, thick,main node/.style={circle, minimum size=4mm, inner sep=0.1mm,draw,font=\tiny\sffamily}]

\node[main node] at (-1,0) (1) {$1_0$};
\node[main node] at (0, {5*\z}) (2) {$1_1$};
\node[main node] at (1,0) (3) {$1_2$};

\node[main node] at (-1,{1*\z} ) (4) {$2_0$};
\node[main node] at (0, {6*\z}) (5) {$2_1$};
\node[main node] at (1,{1*\z}) (6) {$2_2$};

\node[main node] at (-1,{2*\z}) (7) {$3_0$};
\node[main node] at (0, {7*\z}) (8) {$3_1$};
\node[main node] at (1,{2*\z}) (9) {$3_2$};

\node[main node] at (-1,{3*\z}) (10) {$4_0$};
\node[main node] at (0, {8*\z}) (11) {$4_1$};
\node[main node] at (1,{3*\z}) (12) {$4_2$};

\node[main node] at (-1,{4*\z}) (13) {$5_0$};
\node[main node] at (0, {9*\z}) (14) {$5_1$};
\node[main node] at (1,{4*\z}) (15) {$5_2$};

 \path[every node/.style={font=\sffamily}]
(2) edge (1)
(2) edge (3)
(5) edge (4)
(5) edge (6)
(8) edge (7)
(8) edge (9)
(11) edge (10)
(11) edge (12)
(14) edge (13)
(14) edge (15)
(1) edge (6)
(1) edge (9)
(1) edge (12)
(1) edge (15)
(4) edge (3)
(4) edge (9)
(4) edge (12)
(4) edge (15)
(7) edge (3)
(7) edge (6)
(7) edge (12)
(7) edge (15)
(10) edge (3)
(10) edge (6)
(10) edge (9)
(10) edge (15)
(13) edge (3)
(13) edge (6)
(13) edge (9)
(13) edge (12);
\end{tikzpicture}

\\
$H_3$ & $H_4$ & $H_5$ 
\end{tabular}
\caption{Alternative drawings of the graphs $H_k$}\label{figH_k2}
\end{center}
\end{figure}
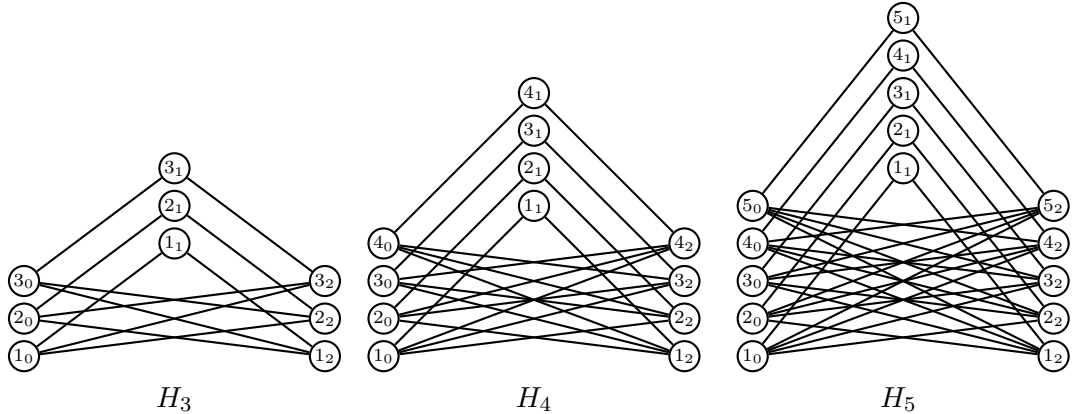

Given a graph $G$, we say that $\sigma : V(G) \to V(G)$ is an \emph{automorphism} of $G$ if, for every $i,j \in V(G)$, $\set{i,j} \in E(G)$ if and only if $\set{\sigma(i), \sigma(j)} \in E(G)$. Notice that the graphs $H_k$ have very rich symmetries, and we mention two automorphisms of $H_k$ that are of particular interest. Define $\sigma_1 : V(H_k) \to V(H_k)$ where, for every $p \in \set{0,1,2}$,
\begin{equation}\label{eqH_ksigma1}
\sigma_1\left( j_p \right) \ce
\begin{cases}
(j+1)_p & \tn{if $1 \leq j \leq k-1$;}\\
1_p & \tn{if $j = k$.}
\end{cases}
\end{equation}
Also define $\sigma_2 : V(H_k) \to V(H_k)$ where 
\begin{equation}\label{eqH_ksigma2}
\sigma_2\left( j_p \right) \ce j_{(2-p)}~\tn{for every $j \in [k]$ and $p \in \set{0,1,2}$.}
\end{equation}
Visually, $\sigma_1$ corresponds to rotating the drawings of $H_k$ in Figure~\ref{figH_k} counterclockwise by $\frac{2\pi}{k}$, and $\sigma_2$ corresponds to reflecting the drawings of $H_k$ in Figure~\ref{figH_k2} along the centre vertical line. Also, notice that $H_k \ominus i$ is either isomorphic to $H_{k-1}$ (if $i \in [k]_1$), or is bipartite (if $i \in [k]_0 \cup [k]_2$). As we shall see, these properties are very desirable in our subsequent analysis of $H_k$. 

Given a graph $G$, let $\a(G)$ denote the maximum cardinality of a stable set in $G$. Notice that since $H_2$ is the $6$-cycle, $\a(H_2)=3$, and the maximum cardinality stable sets of $H_2$ are $\{1_0, 1_2, 2_1\}$ and $\{1_1, 2_0,2_2\}$. Moreover, due to the simple recursive structure of this family of graphs, we can construct stable sets for $H_k$ from stable sets for $H_{k-1}$ for every integer $k \geq 3$. If $S$ is a (maximum-cardinality) stable set for $H_{k-1}$ then $S \cup \{k_1\}$ is a (maximum-cardinality) stable set for $H_k$. This shows that $\a(H_k) = k+1$ for every $k \geq 2$. 

Also, notice that each of the sets $[k]_0$, $[k]_1$, and $[k]_2$ is a stable set in $H_k$. While they each have cardinality $k$ and thus are not maximum cardinality stable sets in $H_k$, they are inclusion-wise maximal (i.e., each of them is not a proper subset of another stable set in $H_k$). The following result characterizes all inclusion-wise maximal stable sets in $H_k$.


\begin{lemma}\label{lemHkStableSets}
Let $k \geq 2$ and let $S \subseteq V(H_k)$ be an inclusion-wise maximal stable set in $H_k$. Then one of the following is true:
\begin{itemize}
\item[(i)]
$|S| = k$, $|S \cap \set{j_0, j_1, j_2}| = 1$ for every $j \in [k]$, and either $S \cap [k]_0 = \emptyset$ or $S \cap [k]_2 = \emptyset$;
\item[(ii)]
$|S| = k+1$, and there exists $j \in [k]$ where
\[
S = \left( [k]_1  \setminus \set{j_1} \right) \cup \set{j_0, j_2}.
\]
\end{itemize}
\end{lemma}

\begin{proof}
First, notice that if $|S \cap \set{j_0,j_1,j_2}| =0$ for some $j \in [k]$, then $S \cup \set{j_1}$ is a stable set, contradicting the maximality of $S$. Thus, we assume that $|S \cap \set{j_0,j_1,j_2}| \geq 1$ for every $j \in [k]$.

If $|S \cap \set{j_0,j_1,j_2}| = 1$ for all $j \in [k]$, then $|S| = k$. Also, since $\set{j_0, j_2'}$ is an edge for every distinct $j,j' \in [k]$, we see that $S \cap [k]_0$ and $S \cap [k]_2$ cannot both be non-empty. Thus, $S$ belongs to (i) in this case.

Next, suppose there exists $j\in [k]$ where $|S \cap \set{j_0,j_1,j_2}| \geq 2$. Then it must be that $j_0,j_2 \in S$ and $j_1 \not\in S$. Then it follows that, for all $j' \neq j$, $S \cap \set{j_0', j_1', j_2'} = \set{j_1'}$, and $S$ belongs to (ii).
\end{proof}

Next, we describe two families of valid inequalities of $\STAB(H_k)$ that are of particular interest. Given distinct indices $j,j' \in [k]$, define
\[
B_{j,j'} \ce V(H_k) \setminus \set{j_1, j_2, j_0', j_1'}. 
\]
Then we have the following.

\begin{lemma}\label{lem61}
For every integer $k \geq 2$,
\begin{itemize}
\item[(i)]
the linear inequality
\begin{equation}\label{lem61eq0}
\sum_{i \in B_{j,j'}} x_{i} \leq k-1
\end{equation}
is a facet of $\STAB(H_k)$ for every pair of distinct $j,j' \in [k]$.
\item[(ii)]
the linear inequality
\begin{equation}\label{lem62eq0}
\sum_{i \in [k]_0 \cup [k]_2} (k-1)x_i + \sum_{i \in [k]_1} (k-2)x_i  \leq k(k-1)
\end{equation}
is valid for $\STAB(H_k)$.
\end{itemize}
\end{lemma}

\begin{proof}
We first prove (i) by induction on $k$. When $k=2$,~\eqref{lem61eq0} gives an edge inequality, which is indeed a facet of $\STAB(H_2)$ since $H_2$ is the $6$-cycle.

Next, assume $k \geq 3$. By the symmetry of $H_k$, it suffices to prove the claim for the case $j=1$ and $j'=2$. First, it follows from Lemma~\ref{lemHkStableSets} that $B_{1,2}$ does not contain a stable set of size $k$, and so~\eqref{lem61eq0} is valid for $\STAB(H_k)$. Next, by the inductive hypothesis, 
\begin{equation}\label{lem61eq1}
\left( \sum_{ i \in B_{1,2}} x_i \right) - \left( x_{k_0} + x_{k_1} + x_{k_2} \right) \leq k-2
\end{equation}
is a facet of $\STAB(H_{k-1})$, and so there exist stable sets $S_1, \ldots, S_{3k-3} \subseteq V(H_{k-1})$ whose incidence vectors are affinely independent and all satisfy~\eqref{lem61eq1} with equality. We then define $S_i' \ce S_i \cup \set{k_1}$ for all $i \in [3k-3]$, $S_{3k-2}' \ce [k]_0, S_{3k-1}' \ce [k]_2$, and $S_{3k}' \ce [k-1]_1 \cup \set{k_0,k_2}$. Then we see that the incidence vectors of $S_1', \ldots, S_{3k}'$ are affinely independent, and they all satisfy~\eqref{lem61eq0} with equality.  This finishes the proof of (i).

We next prove (ii). Consider the inequality obtained by summing~\eqref{lem61eq0} over all distinct $j,j' \in [k]$: 
\begin{equation}\label{lem62eq1}
\sum_{(j,j') \in [k]^2, j \neq j'} \left( \sum_{i \in B_{j,j'}} x_{i} \right) \leq \sum_{(j,j') \in [k]^2, j \neq j'} (k-1).
\end{equation}
Now, the right hand side of~\eqref{lem62eq1} is $k(k-1)(k-1)$. On the other hand, since $|B_{j,j'} \cap [k]_0| =|B_{j,j'} \cap [k]_2| = k-1$ for all $j,j'$, we see that if $i \in [k]_0 \cup [k]_2$, then $x_{i}$ has coefficient $(k-1)(k-1)$ in the left hand side of~\eqref{lem62eq1}. A similar argument shows that $x_{i}$ has coefficient $(k-1)(k-2)$ for all $i \in [k]_1$. Thus,~\eqref{lem62eq0} is indeed $\frac{1}{k-1}$ times~\eqref{lem62eq1}. Therefore,~\eqref{lem62eq0} is a non-negative linear combination of inequalities of the form~\eqref{lem61eq0}, so it follows from (i) that~\eqref{lem62eq0} is valid for $\STAB(H_k)$.
\end{proof}

\subsection{Working from the shadows to prove lower bounds on $\LS_+$-rank}\label{sec32}

Next, we aim to exploit the symmetries of $H_k$ to help simplify our analysis of its $\LS_+$-relaxations. Before we do that, we describe the broader framework of this reduction that shall also be useful in analyzing lift-and-project relaxations in other settings. Given a graph $G$, let $\Aut(G)$ denote the automorphism group of $G$. We also let $\chi_S$ denote the incidence vector of a set $S$. Then we define the notion of $\A$-balancing automorphisms.

\begin{definition}\label{defnABalancing}
Given a graph $G$ and $\A \ce \set{ A_1, \ldots, A_{L}}$ a partition of $V(G)$, we say that a set of automorphisms $\S \subseteq \Aut(G)$ is \emph{$\A$-balancing} if
\begin{enumerate}
\item
For every $\ell \in [L]$, and for every $\sigma \in \S$, $\set{ \sigma(i) : i \in A_{\ell}} = A_{\ell}$. 
\item
For every $\ell \in [L]$, the quantity $|\set{ \sigma \in \S : \sigma(i) =j }|$ is invariant under the choice of $i, j \in A_{\ell}$.
\end{enumerate}
\end{definition}

In other words, if $\S$ is $\A$-balancing, then the automorphisms in $\S$ only map vertices in $A_{\ell}$ to vertices in $A_{\ell}$ for every $\ell \in [L]$. Moreover, for every $i \in A_{\ell}$, the $|\S|$ images of $i$ under automorphisms in $\S$ spread over $A_{\ell}$ evenly, with $|\set{ \sigma \in \S : \sigma(i) = j}| = \frac{|\S|}{|A_{\ell}|}$ for every $j \in A_{\ell}$.

For example, for the graph $H_k$, consider the vertex partition $\A_1 \ce \set{ [k]_0, [k]_1, [k]_2}$ and $\S_1 \ce \set{ \sigma_1^j : j \in [k]}$ (where $\sigma_1$ is as defined in~\eqref{eqH_ksigma1}). Then observe that, for every $p \in \set{0,1,2}$, $\set{ \sigma(i) : i \in [k]_p } = [k]_p$, and 
\[
| \set{ \sigma \in \S_1 : \sigma(i) = j} | = 1
\]
for every $i,j \in [k]_p$. Thus, $\S_1$ is $\A_1$-balancing. Furthermore, if we define $\A_2 \ce \set{ [k]_0 \cup [k]_2 , [k]_1}$ and 
\begin{equation}\label{eqS2balancing}
\S_2 \ce \set{ \sigma_1^j \circ \sigma_2^{j'} : j \in [k], j' \in [2]}
\end{equation}
(where $\sigma_2$ is as defined in~\eqref{eqH_ksigma2}), one can similarly show that $\S_2$ is $\A_2$-balancing. 

Next, we prove several lemmas about $\A$-balancing automorphisms that are relevant to the analysis of $\LS_+$-relaxations. Given $\sigma\in \Aut(G)$, we extend the notation to refer to the function $\sigma : \mR^{V(G)} \to \mR^{V(G)}$ where $\sigma(x)_{i} = x_{\sigma(i)}$ for every $i \in V(G)$. The following lemma follows readily from the definition of $\A$-balancing automorphisms.

\begin{lemma}\label{lemABalancing0}
Let $G$ be a graph, $\A \ce \set{ A_1, \ldots, A_{L}}$ be a partition of $V(G)$, and $\S \subseteq \Aut(G)$ be an $\A$-balancing set of automorphisms. Then, for every $x \in \mR^{V(G)}$,
\[
\frac{1}{|\S|} \sum_{\sigma \in \S} \sigma(x) =\sum_{\ell = 1}^L \frac{1}{|A_{\ell}|} \left( \sum_{i \in A_{\ell}} x_i \right) \chi_{A_{\ell}}.
\]
\end{lemma}

\begin{proof}
For every $\ell \in [L]$ and for every $i \in A_{\ell}$, the fact that $\S$ is $\A$-balancing implies
\begin{equation}\label{lemABalancing0eq1}
\frac{1}{|\S|} \sum_{\sigma \in \S} \sigma(e_i) =  \frac{1}{|A_{\ell}|} \chi_{A_{\ell}}.
\end{equation}
Since $x = \sum_{ i \in V(G)} x_i e_i$, the claim follows by summing $x_i$ times~\eqref{lemABalancing0eq1} over all $i \in V(G)$.
\end{proof}

\begin{lemma}\label{lemsigmax}
Let $G$ be a graph, $\sigma \in \Aut(G)$ be an automorphism of $G$, and $p \geq 0$ be an integer. If $x \in \LS_+^p(G)$, then $\sigma(x) \in \LS_+^p(G)$.
\end{lemma}

\begin{proof}
When $p=0$, $\LS_+^0(G) = \FRAC(G)$, and the claim holds due to $\sigma$ being an automorphism of $G$. Next, it is easy to see from the definition of $\LS_+$ that the operator is invariant under permutation of coordinates (i.e., given $P \subseteq [0,1]^n$, if we let $P_{\sigma} \ce \set{ \sigma(x) : x \in P}$, then $x \in \LS_+(P) \Rightarrow \sigma(x) \in \LS_+(P_{\sigma})$). Applying this insight recursively proves the claim for all $p$.
\end{proof}

Combining the above results, we obtain the following.

\begin{proposition}\label{propABalancing}
Suppose $G$ is a graph, $\A \ce \set{A_1, \ldots, A_{L}}$ is a partition of $V(G)$, and $\S \subseteq \Aut(G)$ is $\A$-balancing. Let $p \geq 0$ be an integer. If $x \in \LS_+^p(G)$, then
\[
x' \ce \sum_{\ell=1}^{L}  \left( \frac{ 1 }{|A_{\ell}|} \sum_{i \in A_{\ell}}x_i \right)\chi_{A_{\ell}}
\]
also belongs to $\LS_+^p(G)$.
\end{proposition}

\begin{proof}
Since $\S$ is $\A$-balancing, it follows from Lemma~\ref{lemABalancing0} that
\[
x' = \frac{1}{|\S|} \sum_{\sigma \in \S} \sigma(x).
\]
Also, since $x \in \LS_+^p(G)$, Lemma~\ref{lemsigmax} implies that $\sigma(x) \in \LS_+^p(G)$ for every $\sigma \in \S$. Thus, $x'$ is a convex combination of points in $\LS_+^p(G)$, which is a convex set. Hence, it follows that $x' \in \LS_+^p(G)$.
\end{proof}

Notice that the symmetrized vector $x'$ in Proposition~\ref{propABalancing} has at most $L$ distinct entries, one for each of $A_{\ell} \in \A$. Thus, instead of fully analyzing a family of SDPs in $\mathbb{S}_+^{\Omega(n^p)}$ or its projections $\LS_+^p(G)$, the presence of $\A$-balancing automorphisms allows us to work with a spectrahedral shadow in $[0,1]^L$, a set of much lower dimension, for a part of the analysis. For instance, in the extreme case when $G$ is vertex-transitive, we see that the entire automorphism group $\Aut(G)$ is $\set{V(G)}$-balancing, and so for every $x \in \LS_+^p(G)$, Proposition~\ref{propABalancing} implies that $\frac{1}{|V(G)|}\left(\sum_{i \in V(G)} x_i \right) \bar{e} \in \LS_+^p(G)$, where $\bar{e}$ denotes the vector of all ones.

Now we turn our focus back to the graphs $H_k$. The presence of an $\A_2$-balancing set of automorphisms (as described in~\eqref{eqS2balancing}) motivates the study of points in $\LS_+^p(H_k)$ of the following form. 

\begin{definition}\label{defnwkab}
Given real numbers $a,b \in \mR$ and an integer $k \geq 2$, $w_k(a,b) \in \mR^{V(H_k)}$ is defined as the vector with entries
\[
[w_k(a,b)]_i \ce \begin{cases}
a & \tn{if $i \in [k]_0 \cup [k]_2$;}\\
b & \tn{if $i \in [k]_1$.}
\end{cases}
\]
\end{definition}

For an example, the inequality~\eqref{lem62eq0} can be rewritten as $w(k-1,k-2)^{\top} x \leq k(k-1)$. The following is a main reason why we are interested in looking into points of the form $w_k(a,b)$.

\begin{lemma}\label{lemwkab0}
Suppose there exists $x \in \LS_+^p(H_k)$ where $x$ violates~\eqref{lem62eq0}. Then there exist real numbers $a,b$ where $w_k(a,b) \in \LS_+^p(H_k)$ and $w_k(a,b)$ violates~\eqref{lem62eq0}.
\end{lemma}

\begin{proof}
Given $x$, let $a \ce \frac{1}{2k}\sum_{i \in [k]_0 \cup [k]_2} x_{i}$ and $b \ce  \frac{1}{k}\sum_{i \in [k]_1} x_{i}$. Due to the presence of the $\A_2$-balancing automorphisms $\S_2$, as well as Proposition~\ref{propABalancing}, we know that $x' \ce w_k(a,b)$ belongs to $\LS_+^{p}(H_k)$. Now since $x$ violates~\eqref{lem62eq0}, 
\[
k(k-1) < w(k-1,k-2)^{\top} x = (k-1)(2ka) + (k-2)(kb) = w(k-1,k-2)^{\top} x',
\]
and so $x'$ violates~\eqref{lem62eq0} as well.
\end{proof}

A key ingredient in our proof of the main result is to find a point $x \in \LS_+^p(H_k)$ where $x$ violates~\eqref{lem62eq0} for some $p \in \Theta(k)$, which would imply that $r_+(H_k) > p$. Lemma~\ref{lemwkab0} assures that, due to the symmetries of $H_k$, we are not sacrificing any sharpness of the result by only looking for such points $x$ of the form $w_k(a,b)$. This enables us to capture important properties of $\LS_+^p(H_k)$ by analyzing a corresponding ``shadow'' of the set in $\mR^2$. More explicitly, given $P \subseteq \mR^{V(H_k)}$, we define
\[
\Phi(P) \ce \set{ (a,b) \in \mR^2 : w_k(a,b) \in P}.
\]
For example, it is not hard to see that
\[
\Phi(\FRAC(H_k)) = \conv\left(\set{ (0,0), \left(\frac{1}{2},0\right), \left(\frac{1}{2}, \frac{1}{2}\right), (0,1)}\right)
\]
for every $k \geq 2$. We can similarly characterize $\Phi(\STAB(H_k))$.

\begin{lemma}\label{lem62a}
For every integer $k \geq 2$, we have
\[
\Phi(\STAB(H_k)) = \conv\left(\set{ (0,0), \left(\frac{1}{2},0\right), \left(\frac{1}{k}, \frac{k-1}{k}\right), (0,1)}\right).
\]
\end{lemma}

\begin{proof}
Let $k \geq 2$ be an integer.
Then, the empty set, $[k]_1$, $[k]_0$, and $[k]_2$ are all stable sets in $H_k$. Notice that $\chi_{\emptyset} = w_k(0,0)$ and $\chi_{[k]_1} = w_k(0,1)$, and thus $(0,0)$ and $(0,1)$ are both in $\Phi(\STAB(H_k))$. Also, since $\frac{1}{2} \chi_{[k]_0} + \frac{1}{2} \chi_{[k]_2} = w_k\left(\frac{1}{2},0\right) \in\STAB(H_k)$, we have $\left(\frac{1}{2},0\right) \in \Phi(\STAB(H_k))$. Next, recall from Lemma~\ref{lemHkStableSets} that for every $j \in [k]$,
\[
S_j \ce \left( [k]_1 \setminus \set{j_1} \right) \cup \set{j_0, j_2}
\]
is a stable set of $H_k$. Thus, $\frac{1}{k} \sum_{j=1}^k \chi_{S_j} = w_k\left( \frac{1}{k}, \frac{k-1}{k} \right) \in \STAB(H_k)$, and so $\left(\frac{1}{k},\frac{k-1}{k}\right) \in\Phi(\STAB(H_k))$. Therefore, $\Phi(\STAB(H_k)) \supseteq \conv\left(\set{ (0,0), \left(\frac{1}{2},0\right), \left(\frac{1}{k}, \frac{k-1}{k}\right), (0,1)}\right)$.

On the other hand, for all $(a,b) \in \Phi(\STAB(H_k))$, it follows from Lemma~\ref{lem61} that 
\[
2(k-1)a +(k-2)b \leq k-1
\]
Since $\Phi(\STAB(H_k)) \subseteq \Phi(\FRAC(H_k))$, using our characterization of $\Phi(\FRAC(H_k))$, we deduce $\Phi(\STAB(H_k))$ is contained in the set
\[
\conv\left(\set{ (0,0), \left(\frac{1}{2},0\right), \left(\frac{1}{2}, \frac{1}{2}\right), (0,1)}\right) \cap \left\{(a,b)\in \mR^2  \, : \, 2(k-1)a +(k-2)b \leq k-1 \right\}.
\]
However, the above set is exactly $\conv\left(\set{ (0,0), \left(\frac{1}{2},0\right), \left(\frac{1}{k}, \frac{k-1}{k}\right), (0,1)}\right)$.
Therefore, 
\[
\Phi(\STAB(H_k)) = \conv\left(\set{ (0,0), \left(\frac{1}{2},0\right), \left(\frac{1}{k}, \frac{k-1}{k}\right), (0,1)}\right).
\]
\end{proof}

Even though $\STAB(H_k)$ is an integral polytope, notice that $\Phi(\STAB(H_k))$ is not integral. Nonetheless, it is clear that
\[
\LS_+^p(H_k) = \STAB(H_k) \Rightarrow \Phi(\LS_+^p(H_k)) = \Phi(\STAB(H_k)).
\]
Thus, to show that $r_+(H_k) > p$, it suffices to find a point $(a,b) \in  \Phi(\LS_+^p(H_k)) \setminus \Phi(\STAB(H_k))$. More generally, given a graph $G$ with a set of $\A$-balancing automorphisms where $\A$ partitions $V(G)$ into $L$ sets, one can adapt our approach and study the $\LS_+$-relaxations of $G$ via analyzing $L$-dimensional shadows of these sets.

\subsection{Strategy for the proof of the main result}\label{sec32a}

So far, we have an infinite family of graphs $\left\{H_k\right\}$ with nice symmetries. In Lemma~\ref{lem61}(i)
we derived a family of facets of $\STAB(H_k)$, and then showed in Lemma~\ref{lem61}(ii) that these facets imply a symmetrized valid inequality for $\STAB(H_k)$. Notably, the left hand side of this symmetrized inequality only has two distinct entries --- one for the vertices in $[k]_0 \cup [k]_2$, and the other for vertices in $[k]_1$.

Also, due to the symmetries of $H_k$, these graphs admit $\A$-balancing automorphisms. Using that and Proposition~\ref{propABalancing}, we can use an arbitrary point $x \in \LS_+^p(H_k)$ to derive a symmetrized point $x' \in \LS_+^p(H_k)$ for every $k \geq 2$ and $p \geq 0$. In particular, using the $\A$-balancing automorphisms described in~\eqref{eqS2balancing}, we are assured that the symmetrized $x'$ has at most two distinct entries --- one for vertices in $[k]_0 \cup [k]_2$, and one for vertices in $[k]_1$.
 
These findings motivate the study of the two dimensional shadow $\Phi(\LS_+^p(H_k))$ of $\LS_+^p(H_k)$. We also have a complete characterization of $\Phi(\STAB(H_k))$ for every integer $k\geq 2$ (Lemma~\ref{lem62a}). Thus, to show that $H_k$ has $\LS_+$-rank greater than $p$, it suffices to establish the existence of a point in the two dimensional set $\Phi(\LS_+^p(H_k)) \setminus \Phi(\STAB(H_k))$.

The rest of the proof of the main result, presented in the next section, starts by characterizing all symmetric positive semidefinite matrices that could certify the membership of a vector $x'$ in $\LS_+(H_k)$ ($x'$, as mentioned above, has at most two distinct entries). This characterization is relatively simple, since due to the isolated symmetries of $x'$, there always exists a symmetric positive semidefinite certificate matrix for certifying the membership of $x'$ in $\LS_+(H_k)$ which has at most four distinct entries (ignoring the $00^{\textup{th}}$ entry of the certificate matrix, which is one).  Next, we construct a compact convex set which is described by three linear inequalities and a quadratic inequality such that this set is a subset of $\Phi(\LS_+(H_k))$ and a strict superset of $\Phi(\STAB(H_k))$ for every integer $k \geq 4$. This enables us to conclude for all $k \geq 4$, $r_+(H_k) \geq 2$, and establish some of the tools for the rest of the proof. The point $w_k\left(\frac{1}{k},\frac{k-1}{k}\right)$ already plays a special role. 

Then, we put all these tools together to prove that there is a point in $\LS_+^p(H_k) \setminus \STAB(H_k)$ which is very close to
$w_k\left(\frac{1}{k},\frac{k-1}{k}\right)$ (we do this partly by working in the $2$-dimensional space where $\Phi(\LS_+^p(H_k)$ lives). For the recursive construction of certificate matrices (to establish membership in $\LS_+^p(H_k)$), we show that in addition to
the matrix being symmetric, positive semidefinite, and satisfying a simple linear inequality, membership
of two suitable vectors $w_{k-1}(a_1,b_1)$ and $w_{k-1}(a_2,b_2)$ in $\LS_+^{p-1}(H_{k-1})$ suffice (Lemma~\ref{lem65}). The rest of the analysis proves that there exist values for these parameters which allow the construction of certificate matrices for suitable pairs of integers $k$ and $p$.

\section{Proof of the main result}\label{sec4}
\subsection{$\Phi(\LS_+(H_k))$ --- the shadow of the first relaxation}\label{sec41}

Here, we aim to study the set $\Phi(\LS_+(H_k))$. To do that, we first look into potential certificate matrices for $w_k(a,b)$ that have plenty of symmetries. Given $k \in \mN$ and $a,b,c,d \in \mR$, we define the matrix $W_k(a,b,c,d) \ce \begin{bmatrix} 1 & w_k(a,b)^{\top} \\ w_k(a,b) & \overline{W} \end{bmatrix}$, where
\[
\overline{W} \ce \begin{bmatrix} a & 0 & a-c \\ 0 & b & 0 \\ a-c & 0 & a \end{bmatrix} \otimes I_k + 
\begin{bmatrix} c & a-c & 0 \\ a-c & d & a-c \\ 0 & a-c & c \end{bmatrix} \otimes (J_k - I_k).
\]
Note that $\otimes$ denotes the Kronecker product, $I_k$ is the $k$-by-$k$ identity matrix, and $J_k$ is the $k$-by-$k$ matrix of all ones. Also, $W_k(a,b,c,d) \in \mR^{(\set{0} \cup V(H_k)) \times (\set{0} \cup V(H_k))}$, and the $|V_k| = 3k$ columns of $\overline{W}$ are indexed by the vertices $1_0, 1_1, 1_2, 2_0, 2_1, 2_2, \ldots$ from left to right, with the rows following the same ordering. Then we have the following.

\begin{lemma}\label{lem63}
Let $k \in \mN$ and $a,b,c,d \in \mR$. Then $W_k(a,b,c,d) \succeq 0$ if and only if the following conditions hold:
\begin{itemize}
\item[(S1)]
$c \geq 0$;
\item[(S2)]
$a - c \geq 0$;
\item[(S3)]
$(b-d) - (a-c) \geq 0$;
\item[(S4)]
$2a + (k-2)c - 2k a^2 \geq 0$;
\item[(S5)]
$(2a + (k-2)c - 2k a^2)(2b + 2(k-1)d - 2kb^2) - (2(k-1)(a-c) - 2kab)^2 \geq 0$.
\end{itemize}
\end{lemma}

\begin{proof}
Define matrices $\overline{W}_1, \overline{W}_2, \overline{W}_3 \in \mR^{3k \times 3k}$ where
\begin{align*}
\overline{W}_1 & \ce \frac{1}{2} \cdot J_k \otimes \begin{bmatrix} c & 0 & -c \\ 0 & 0 & 0 \\ -c & 0 & c \end{bmatrix}, \\
\overline{W}_2 & \ce \frac{1}{k} \cdot (kI_k - J_k) \otimes \begin{bmatrix} a-c & c-a & a-c \\ c-a & b-d & c-a \\ a-c & c-a & a-c \end{bmatrix}, \\
\overline{W}_3 & \ce \frac{1}{2k} \cdot J_k \otimes \begin{bmatrix} 
2a + (k-2)c - 2k a^2 & 2(k-1)(a-c) - 2kab & 2a + (k-2)c - 2k a^2\\
 2(k-1)(a-c) - 2kab &2b + 2(k-1)d - 2kb^2 & 2(k-1)(a-c) - 2kab\\
2a + (k-2)c - 2k a^2 & 2(k-1)(a-c) - 2kab & 2a + (k-2)c - 2k a^2
\end{bmatrix}.
\end{align*}
Then
\begin{align*}
&\overline{W}_1 + \overline{W}_2 + \overline{W}_3 \\
={}& \begin{bmatrix} a - a^2 & -ab & a-c-a^2 \\ -ab & b -b^2 & -ab \\ a-c-a^2 & -ab & a-ab^2 \end{bmatrix} \otimes I_k + 
\begin{bmatrix} c -a^2 & a-c-ab & -a^2 \\ a-c-ab & d -b^2& a-c-ab \\ -a^2 & a-c-ab & c-a^2 \end{bmatrix} \otimes (J_k - I_k)\\
={}& \overline{W} - w_k(a,b)(w_k(a,b))^{\top},
\end{align*}
which is a Schur complement of $W_k(a,b,c,d)$. Thus, we see that $W_k(a,b,c,d) \succeq 0$ if and only if $\overline{W}_1 + \overline{W}_2 + \overline{W}_3 \succeq 0$. Moreover, observe that the columns of $\overline{W}_i$ and $\overline{W}_j$ are orthogonal whenever $i \neq j$. Thus, $\overline{W}_1 + \overline{W}_2 + \overline{W}_3 \succeq 0$ if and only if $\overline{W}_1, \overline{W}_2$, and $\overline{W}_3$ are all positive semidefinite. Now observe that
\begin{align*}
\overline{W}_1 \succeq 0 &\iff \begin{bmatrix} c & -c \\ -c & c \end{bmatrix} \succeq 0 \iff \tn{(S1) holds},\\
\overline{W}_2 \succeq 0 &\iff \begin{bmatrix} a-c & c-a \\ c-a & b-d \end{bmatrix} \succeq 0 \iff \tn{(S2) and (S3) hold},\\
\overline{W}_3 \succeq 0 &\iff \begin{bmatrix} 2a + (k-2)c - 2k a^2 & 2(k-1)(a-c) - 2kab \\ 2(k-1)(a-c) - 2kab & 2b + 2(k-1)d - 2kb^2\end{bmatrix} \succeq 0 \iff \tn{(S4) and (S5) hold}.
\end{align*}
Thus, the claim follows.
\end{proof}

Next, for convenience, define $q_k \ce 1-\sqrt{\frac{k}{2k-2}}$, and 
\[
p_k(x,y) \ce (2x^2-x)+2q_k^2(y^2-y)+4q_kxy.
\]
Notice that the curve $p_k(x,y) = 0$ is a parabola for all $k \geq 3$. Then, using Lemma~\ref{lem63}, we have the following.

\begin{proposition}\label{prop64}
For every $k \geq 4$,
\begin{equation}\label{prop64eq0}
\Phi(\LS_+(H_k)) \supseteq \set{ (x,y) \in \mR^2: p_k(x,y) \leq 0, x+y \leq 1, x \geq 0, y \geq 0}.
\end{equation}
\end{proposition}

\begin{proof}
For convenience, let $C$ denote the set on the right hand side of~\eqref{prop64eq0}. Notice that the boundary points of the triangle $\set{ (x,y) : x+y \leq 1, x \geq 0,y \geq 0}$ which lie in $C$ are also boundary points of $\Phi(\STAB(H_k))$. Thus, let us define the set of points
\[
C_0 \ce \set{ (x,y) \in \mR^2: p_k(x,y) = 0, \frac{1}{k} < x < \frac{1}{2}}.
\]
To prove our claim, it suffices to prove that for all $(a,b) \in C_0$, there exist $c,d \in \mR$ such that $W_k(a,b,c,d)$ certifies $w_k(a,b) \in \LS_+(H_k)$.

\begin{figure}[ht!]
\begin{center}
\includegraphics[width=5cm]{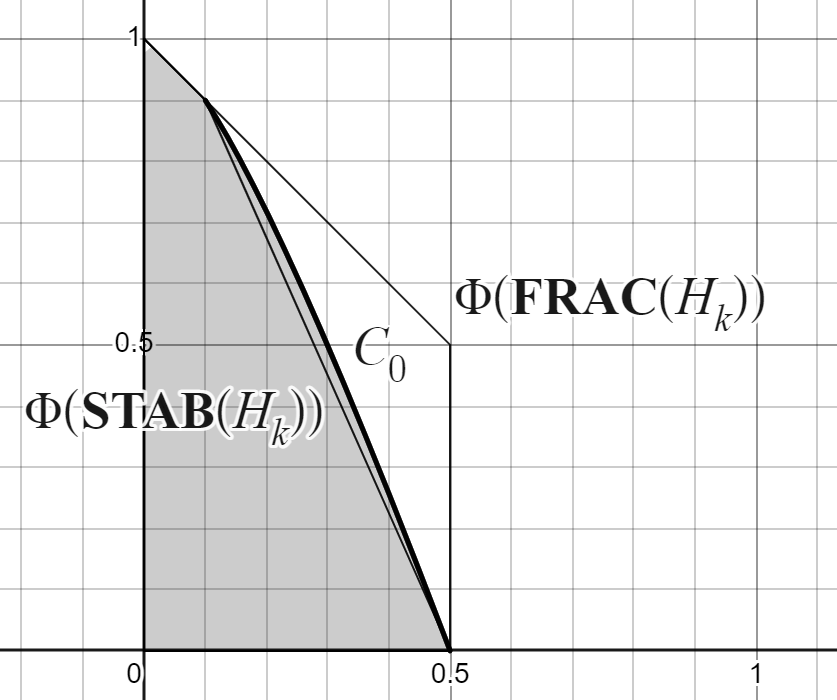}
\end{center}
\caption{Visualizing the set $C$ for the case $k=10$}\label{figC}
\end{figure}

To help visualize our argument, Figure~\ref{figC} (which is produced using Desmos' online graphing calculator~\cite{Desmos}) illustrates the set $C$ for the case of $k=10$. 

Now given $(a, b) \in \mR^2$ (not necessarily in $C_0$), consider the conditions (S3) and (S5) from Lemma~\ref{lem63}:
\begin{align}
\label{prop64eq1} b-a+c &\geq d,\\
\label{prop64eq2} (2a + (k-2)c - 2k a^2)(2b + 2(k-1)d - 2kb^2) - (2(k-1)(a-c) - 2kab)^ 2 & \geq 0.
\end{align}
If we substitute $d = b-a+c$ into~\eqref{prop64eq2} and solve for $c$ that would make both sides equal, we would obtain the quadratic equation $p_2c^2 + p_1c + p_0 =0$ where
\begin{align*}
p_2 \ce {}& (k-2)(2(k-1)) - (-2(k-1))^2,\\
p_1 \ce {}& (k-2)(2b + 2(k-1)(b-a)-2kb^2) + (2a - 2ka^2)(2(k-1))\\
& -2(-2(k-1))(2(k-1)a - 2kab),\\
p_0 \ce {}& (2a - 2ka^2)(2b + 2(k-1)(b-a)-2kb^2) - (2(k-1)a - 2kab)^2.
\end{align*}
We then define
\[
c \ce \frac{-p_1}{2p_2} = -a^2-2ab-\frac{b^{2}}{2}+\frac{3a}{2}+\frac{b}{2}+\frac{b\left(b-1\right)}{2\left(k-1\right)},
\]
and $d \ce b-a+c$. We claim that, for all $(a,b) \in C_0$, $W_k(a,b,c,d)$ would certify $w_k(a,b) \in \LS_+(H_k)$. First, we provide some intuition for the choice of $c$. Let $\overline{q}_k \ce 1+ \sqrt{\frac{k}{2k-2}}$ and 
\[
\overline{p}_k(x,y) \ce (2x^2-x) + 2\overline{q}_k^2(y^2-y) + 4\overline{q}_kxy.
\]
Then, if we consider the discriminant $\Delta p \ce p_1^2 - 4p_0p_2$, one can check that
\[
\Delta p = 4(k-1)^2 p_k(a,b) \overline{p}_k(a,b).
\]
Thus, when $\Delta p > 0$, there would be two solutions to the quadratic equation $p_2x^2 + p_1x + p_0 = 0$, and $c$ would be defined as the midpoint of these solutions. In particular, when $(a,b) \in C_0$, $p_k(a,b) = \Delta p = 0$, and so $c = \frac{-p_1}{2p_2}$ would indeed be the unique solution that satisfies both~\eqref{prop64eq1} and~\eqref{prop64eq2} with equality.

Now we verify that $Y \ce W_k(a,b,c,d)$, as defined, satisfies all the conditions imposed by $\LS_+$. We first show that $W_k(a,b,c,d) \succeq 0$ by verifying the conditions from Lemma~\ref{lem63}. Notice that (S3) and (S5) must hold by the choice of $c$ and $d$. Next, we check (S1), namely $c \geq 0$. Define the region
\[
T \ce \set{ (x,y) \in \mR^2 : \frac{1}{k} \leq x \leq \frac{1}{2}, \frac{(1-2x)(k-1)}{k-2} \leq y \leq 1 - x}.
\]
In other words, $T$ is the triangular region with vertices $\left(\frac{1}{k}, \frac{k-1}{k} \right), \left(\frac{1}{2},0 \right)$, and $\left(\frac{1}{2}, \frac{1}{2} \right)$. Thus, $T$ contains $C_0$ and it suffices to show that $c \geq 0$ over $T$. Fixing $k$ and viewing $c$ as a function of $a$ and $b$, we obtain
\[
\frac{ \partial c}{\partial a} = -2a-2b+\frac{3}{2}, \quad 
\frac{ \partial c}{\partial b} = \frac{(-4a-2b+1)k +4a +4b-2}{2k-2}.
\]
Solving $\frac{ \partial c}{\partial a} = \frac{ \partial c}{\partial b} = 0$, we obtain the unique solution $(a,b) = \left( \frac{-k+2}{4k}, \frac{4k-2}{4k} \right)$, which is outside of $T$. Next, one can check that $c$ is non-negative over the three edges of $T$, and we conclude that $c \geq 0$ over $T$, and thus (S1) holds. The same approach also shows that both $a-c$ and $2a+(k-2)c-2ka^2$ are non-negative over $T$, and thus (S2) and (S4) hold as well, and we conclude that $Y \succeq 0$.

Next, we verify that $Ye_i, Y(e_0-e_i) \in \cone(\FRAC(H_k))$. By the symmetry of $H_k$, it suffices to verify these conditions for the vertices $i=1_0$ and $i=1_1$. 
\begin{itemize}
\item
$Ye_{1_0}$: Define $S_1 \ce \set{1_0, 1_2} \cup \set{i_1 : 2 \leq i \leq k}$ and $S_2 \ce [k]_0$. Observe that both $S_1, S_2$ are stable sets of $H_k$, and that
\begin{equation}\label{prop64Ye1}
Ye_{1_0} = (a-c) \begin{bmatrix} 1 \\ \chi_{S_1} \end{bmatrix} + c \begin{bmatrix} 1 \\ \chi_{S_2} \end{bmatrix},
\end{equation}
Since we verified above that $c \geq 0$ and $a-c \geq 0$, $Ye_{1_0} \in \cone(\STAB(H_k))$, which is contained in $\cone(\FRAC(H_k))$.
\item
$Ye_{1_1}$: The non-trivial edge inequalities imposed by $Ye_{1_1} \in \cone(\FRAC(H_k))$ are
\begin{align}
\label{prop64e1} [Ye_{1_1}]_{2_2} + Y[e_{1_1}]_{3_0} \leq [Ye_{1_1}]_{0} & \Rightarrow 2(a-c) \leq b, \\
\label{prop64e2} [Ye_{1_1}]_{2_0} + [Ye_{1_1}]_{2_1} \leq [Ye_{1_1}]_{0} & \Rightarrow a-c+d \leq b.
\end{align}
Note that~\eqref{prop64e2} is identical to (S3), which we have already established. Next, we know from (S4) that $c \geq \frac{2ka^2-2a}{k-2}$. That together with the fact that $2(k-1)a+(k-2)b \geq k-1$ for all $(a,b) \in C_0$ and $k \geq 4$ implies~\eqref{prop64e1}.
\item
$Y(e_0-e_{1_0})$: The non-trivial edge inequalities imposed by $Y(e_0-e_{1_0}) \in \cone(\FRAC(H_k))$ are
\begin{align}
\label{prop64e3} [Y(e_0-e_{1_0})]_{2_2} + [Y(e_0-e_{1_0})]_{3_0} \leq [Y(e_0-e_{1_0})]_{0} & \Rightarrow a + (a-c) \leq 1-a, \\
\label{prop64e4} [Y(e_0-e_{1_0})]_{2_1} + [Y(e_0-e_{1_0})]_{2_2} \leq [Y(e_0-e_{1_0})]_{0} & \Rightarrow (b-a+c)+ a \leq 1-a.
\end{align}
\eqref{prop64e3} follows from \eqref{prop64e1} and the fact that $a -c \geq 0$. For~\eqref{prop64e4}, we aim to show that $a+b+c \leq 1$. Define the quantity
\[
g(x,y) \ce 1-\frac{5}{2}x-\frac{3}{2}y+x^{2}+\frac{1}{2}y^{2}+2xy-\frac{y\left(y-1\right)}{2\left(k-1\right)}.
\]
Then $g(a,b) = 1-a-b-c$. Notice that, for all $k$, the curve $g(x,y) = 0$ intersects with $C$ at exactly three points: $(0,1), \left( \frac{1}{k}, \frac{k-1}{k} \right)$, and $\left(\frac{1}{2} , 0\right)$. In particular, the curve does not intersect the interior of $C$. Therefore, $g(x,y)$ is either non-negative or non-positive over $C$. Since $g(0,0) = 1$, it is the former. Hence, $g(x,y) \geq 0$ over $C$ (and hence $C_0$), and~\eqref{prop64e4} holds.
\item
$Y(e_0-e_{1_1})$: The non-trivial edge inequalities imposed by $Y(e_0-e_{1_1}) \in \cone(\FRAC(H_k))$ are
\begin{align}
\label{prop64e5} [Y(e_0-e_{1_1})]_{1_0} + [Y(e_0-e_{1_1})]_{2_2} \leq [Y(e_0-e_{1_1})]_{0} & \Rightarrow a + c \leq 1-b, \\
\label{prop64e6} [Y(e_0-e_{1_1})]_{2_0} + [Y(e_0-e_{1_1})]_{2_1} \leq [Y(e_0-e_{1_1})]_{0} & \Rightarrow c + (b-d) \leq 1-b.
\end{align}
\eqref{prop64e5} is identical to~\eqref{prop64e3}, which we have verified above. Finally,~\eqref{prop64e6} follows from (S3) and the fact that $a+b \leq 1$ for all $(a,b) \in C$.
\end{itemize}

This completes the proof.
\end{proof}

An immediate consequence of Proposition~\ref{prop64} is the following.

\begin{corollary}\label{corH_4}
For all $k \geq 4$, $r_+(H_k) \geq 2$.
\end{corollary}

\begin{proof}
For every $k \geq 4$, the set described in~\eqref{prop64eq0} is not equal to $\Phi(\STAB(H_k))$. Thus, there exists $w_k(a,b) \in \LS_+(H_k) \setminus \STAB(H_k)$ for all $k \geq 4$, and the claim follows.
\end{proof}

Corollary~\ref{corH_4} is sharp --- notice that destroying any vertex in $H_3$ yields a bipartite graph, so it follows from Theorem~\ref{thmDeleteDestroy}(i) that $r_+(H_3)=1$. Also, since destroying a vertex in $H_4$ either results in $H_3$ or a bipartite graph, we see that $r_+(H_4) = 2$.

\subsection{Showing $r_+(H_k) = \Theta(k)$}\label{sec42}

We now develop a few more tools that we need to establish the main result of this section. Again, to conclude that $r_+(H_k) > p$, it suffices to show that $\Phi(\LS_+^p(H_k)) \supset \Phi(\STAB(H_k))$. In particular, we will do so by finding a point in $\Phi(\LS_+^p(H_k)) \setminus \Phi(\STAB(H_k))$ that is very close to the point $\left(\frac{1}{k}, \frac{k-1}{k}\right)$. Given $(a,b) \in \mR^2$, let
\[
s_k(a,b) \ce \frac{\frac{k-1}{k} - b}{\frac{1}{k} - a}.
\]
That is, $s_k(a,b)$ is the slope of the line that contains the points $(a,b)$ and $\left(\frac{1}{k}, \frac{k-1}{k}\right)$. Next, define
\[
f(k, p) \ce \sup \set{ s_k(a,b) : (a,b) \in \Phi(\LS_+^p(H_k)), a > \frac{1}{k}}.
\]
In other words, $f(k,p)$ is the slope of the tangent line to $\Phi(\LS_+^p(H_k))$ at the point $\left(\frac{1}{k}, \frac{k-1}{k}\right)$ towards the right hand side. Thus, for all $\ell < f(k,p)$, there exists $\varepsilon > 0$ where the point $\left( \frac{1}{k} + \varepsilon, \frac{k-1}{k} + \ell \varepsilon \right)$ belongs to $\Phi(\LS_+^p(H_k))$. For $p=0$ (and so $\LS_+^p(H_k) = \FRAC(H_k)$), observe that $f(k, 0) = -1$ for all $k \geq 2$ (attained by the point $\left( \frac{1}{2}, \frac{1}{2} \right))$. Next, for $p=1$, consider the polynomial $p_k(x,y)$ defined before Proposition~\ref{prop64}. Then any point $(x,y)$ on the curve $p_k(x,y) = 0$ has slope
\[
\frac{\partial}{\partial x} p_k(x,y) = \frac{1-4x- 4q_k y}{4q_k^2y - q_k + 4q_kx}.
\]
Thus, by Proposition~\ref{prop64},
\begin{equation}\label{dydx:eq1}
f(k,1) \geq \left. \frac{\partial}{\partial x} p_k(x,y) \right|_{(x,y) = \left(\frac{1}{k}, \frac{k-1}{k}\right)} = 
-1 - \frac{k}{3k^2- 2(k-1)^2 \sqrt{ \frac{2k}{k-1}} -4k}
\end{equation}
for all $k \geq 4$. Finally, if $p \geq r_+(H_k)$, then $f(k,p) = -\frac{2(k-1)}{k-2}$ (attained by the point $\left( \frac{1}{2},0 \right) \in \Phi(\STAB(H_k))$).

We will prove our $\LS_+$-rank lower bound on $H_k$ by showing that $f(k,p) > -\frac{2(k-1)}{k-2}$ for some $p = \Theta(k)$. To do so, we first show that the recursive structure of $H_k$ allows us to establish $(a,b) \in \Phi(\LS_+^p(H_k))$ by verifying (among other conditions) the membership of two particular points in $\Phi(\LS_+^{p-1}(H_{k-1}))$, which will help us relate the quantities $f(k-1,p-1)$ and $f(k,p)$. Next, we bound the difference $f(k-1,p-1) - f(k,p)$ from above, which implies that it takes $\LS_+$ many iterations to knock the slopes $f(k,p)$ from that of $\Phi(\FRAC(H_k))$ down to that of $\Phi(\STAB(H_k))$.

First, here is a tool that will help us verify certificate matrices recursively.

\begin{lemma}\label{lem65}
Suppose $a,b,c,d \in \mR$ satisfy all of the following:
\begin{itemize}
\item[(i)]
$W_k(a,b,c,d) \succeq 0$,
\item[(ii)]
$2b+ 2c -d \leq 1$,
\item[(iii)]
$w_{k-1}\left(\frac{a-c}{b}, \frac{d}{b} \right), w_{k-1}\left(\frac{a-c}{1-a-c}, \frac{b-a+c}{1-a-c}\right) \in \LS_+^{p-1}(H_{k-1})$.
\end{itemize}
Then $W_k(a,b,c,d)$ certifies $w_k(a,b) \in \LS_+^p(H_k)$.
\end{lemma}

\begin{proof}
For convenience, let $Y \ce W_k(a,b,c,d)$. First, $Y \succeq 0$ from (i). Next, we focus on the following column vectors:

\begin{itemize}
\item
$Ye_{1_0}$: $Y \succeq 0$ implies that $c \geq 0$ and $a-c \geq 0$ by Lemma~\ref{lem63}. Then it follows from~\eqref{prop64Ye1} that $Ye_{1_0} \in \cone(\STAB(H_k)) \subseteq \cone(\LS_+^{p-1}(H_k))$.
\item
$Ye_{1_1}$: (iii) implies $\begin{bmatrix} b \\ w_{k-1}(a-c, d) \end{bmatrix} \in \cone(\LS_+^{p-1}(H_{k-1}))$. Thus,
\[
Ye_{1_1}= \begin{bmatrix} b \\ 0 \\ b \\ 0 \\ w_{k-1}(a-c,d) \end{bmatrix} \in \cone(\LS_+^{p-1}(H_k)).
\]
\item
$Y(e_0-e_{1_0})$: Let $S_1 \ce [k]_2$, which is a stable set in $H_k$. Then observe that
\[
Y(e_0 - e_{1_0}) = c \begin{bmatrix} 1 \\ \chi_{S_1} \end{bmatrix} + \begin{bmatrix} 1 - a -c \\ 0 \\ b \\ 0 \\ w_{k-1}(a-c, b-a+c) \end{bmatrix}.
\]
By (iii) and the fact that $\cone(\LS_+^{p-1}(H_k))$ is a convex cone, it follows that $Y(e_0-e_{1_0}) \in \cone(\LS_+^{p-1}(H_k))$.
\item
$Y(e_0-e_{1_1})$: Define $S_2 \ce [k]_0, S_3 \ce \set{1_0, 1_2} \cup \set{i_1 : 2 \leq i \leq k}$, and $S_4 \ce \set{i_1 : 2 \leq i \leq k}$, which are all stable sets in $H_k$. Now observe that
\begin{align*}
Y(e_0 - e_{1_1}) {}=& c \begin{bmatrix} 1 \\ \chi_{S_1} \end{bmatrix}+ c \begin{bmatrix} 1 \\ \chi_{S_2} \end{bmatrix} + (a-c) \begin{bmatrix} 1 \\ \chi_{S_3} \end{bmatrix} + (b-d-a+c) \begin{bmatrix} 1 \\ \chi_{S_4} \end{bmatrix} \\ 
& + (1-2b- 2c+d)\begin{bmatrix} 1 \\ 0 \end{bmatrix}.
\end{align*}
Since $Y \succeq 0$, $b-d-a+c \geq 0$ from (S3). Also, $1-2b-2c+d \geq 0$ by (ii). Thus, $Y(e_0-e_{1_1})$ is a sum of vectors in $\cone(\STAB(H_k))$, and thus belongs to $\cone(\LS_+^{p-1}(H_k))$.
\end{itemize}
By the symmetry of $H_k$ and $W_k(a,b,c,d)$, it suffices to verify the membership conditions for the above columns. Thus, it follows that $W_k(a,b,c,d)$ indeed certifies $w_k(a,b) \in \LS_+^p(H_k)$.
\end{proof}

\begin{example}\label{egH7}
We illustrate Lemma~\ref{lem65} by using it to show that $r_+(H_7) \geq 3$. Let $k = 7, a = 0.1553, b = 0.8278, c = 0.005428$, and $d = 0.6665$. Then one can check (via Lemma~\ref{lem63}) that $W_k(a,b,c,d) \succeq 0$, and $2b+2c-d \leq 1$. Also, one can check that $w_{k-1}\left(\frac{a-c}{b}, \frac{d}{b}\right)$ and $w_{k-1}\left(\frac{a-c}{1-a-c}, \frac{b-a+c}{1-a-c}\right)$ both belong to $\LS_+(H_{k-1})$ using Proposition~\ref{prop64}. Thus, Lemma~\ref{lem65} applies, and $w_k(a,b) \in \LS_+^2(H_k)$. Now observe that $2(k-1)a +(k-2)b = 6.0026 > k-1$, and so $w_k(a,b) \not\in \STAB(H_k)$, and we conclude that $r_+(H_7) \geq 3$.
\end{example}

Next, we apply Lemma~\ref{lem65} iteratively to find a lower bound for the $\LS_+$-rank of $H_k$ as a function of $k$. The following is an updated version of Lemma~\ref{lem65} that gets us a step closer to directly relating $f(k,p)$ and $f(k-1,p-1)$.

\begin{lemma}\label{lem67}
Suppose $a,b,c,d \in \mR$ satisfy all of the following:
\begin{itemize}
\item[(i)]
$W_k(a,b,c,d) \succeq 0$,
\item[(ii)]
$2b+ 2c -d \leq 1$,
\item[(iii)]
$\max\set{ s_{k-1}\left(\frac{a-c}{b}, \frac{d}{b} \right), s_{k-1}\left(\frac{a-c}{1-a-c}, \frac{b-a+c}{1-a-c} \right)} \leq f(k-1, p-1)$.
\end{itemize}
Then $f(k,p) \geq s_k(a,b)$.
\end{lemma}

\begin{proof}
Given $a,b,c,d \in \mR$ that satisfy the given assumptions, define
\begin{align*}
a(\ld) & \ce \frac{\ld}{k} + (1-\ld)a, &
b(\ld) & \ce \frac{\ld(k-1)}{k} + (1-\ld)b, \\
c(\ld) & \ce (1-\ld)c, &
d(\ld) & \ce \frac{\ld(k-2)}{k} + (1-\ld)d.
\end{align*}
Then notice that
\begin{equation}\label{lem67eq1}
W_k(a(\ld),b(\ld),c(\ld),d(\ld)) = \ld W_k\left( \frac{1}{k}, \frac{k-1}{k}, 0, \frac{k-2}{k} \right) + (1-\ld) W_k(a,b,c,d).
\end{equation}
Observe (e.g., via Lemma~\ref{lem63}) that $W_k\left( \frac{1}{k}, \frac{k-1}{k}, 0, \frac{k-2}{k} \right) \succeq 0$ for all $k \geq 2$. Since $W_k(a,b,c,d) \succeq 0$ from (i), it follows from~\eqref{lem67eq1} and the convexity of the positive semidefinite cone that 
\[
W_k(a(\ld),b(\ld),c(\ld),d(\ld)) \succeq 0
\]
for all $\ld \in [0,1]$.

Now observe that for all $\ld > 0$, $s_k(a(\ld), b(\ld)) = s_k(a,b)$, $s_{k-1}\left(\frac{a(\ld)-c(\ld)}{b(\ld)}, \frac{d(\ld)}{b(\ld)} \right) = s_{k-1}\left(\frac{a-c}{b}, \frac{d}{b} \right)$, and $s_{k-1}\left(\frac{a(\ld)-c(\ld)}{1-a(\ld)-c(\ld)}, \frac{b(\ld)-a(\ld)+c(\ld)}{1-a(\ld)-c(\ld)} \right)= s_{k-1}\left(\frac{a-c}{1-a-c}, \frac{b-a+c}{1-a-c} \right)$. By assumption (iii), there must be a sufficiently small $\ld >0$ where $w_{k-1}\left(\frac{a(\ld)-c(\ld)}{b(\ld)}, \frac{d(\ld)}{b(\ld)} \right)$ and $w_{k-1}\left(\frac{a(\ld)-c(\ld)}{1-a(\ld)-c(\ld)}, \frac{b(\ld)-a(\ld)+c(\ld)}{1-a(\ld)-c(\ld)} \right)$ are both contained in $\LS_+^{p-1}(H_k)$. Then Lemma~\ref{lem65} implies that $w_k( a(\ld), b(\ld)) \in \LS_+^p(H_k)$, and the claim follows.
\end{proof}

Next, we define four values corresponding to each $k$ that will be important in our subsequent argument:
\begin{align*}
u_1(k) & \ce -\frac{2(k-1)}{k-2}, &
u_2(k) & \ce \frac{k-4- \sqrt{17k^2-48k+32}}{2(k-2)},\\
u_3(k) & \ce \frac{4(k-1)(-3k+4-2\sqrt{k-1}) }{(k-2)(9k-10)}, &
u_4(k) & \ce -1 - \frac{k}{3k^2- 2(k-1)^2 \sqrt{ \frac{2k}{k-1}} -4k}.
\end{align*}

Notice that $u_1(k) = f(k,p)$ for all $p \geq r_+(H_k)$, and $u_4(k)$ is the expression given in~\eqref{dydx:eq1}, the lower bound for $f(k,1)$ that follows from Proposition~\ref{prop64}. Then we have the following.

\begin{lemma}\label{lem69}
For every $k \geq 5$,
\[
u_1(k) < u_2(k) < u_3(k) < u_4(k).
\]
\end{lemma}

\begin{proof}
First, one can check that the chain of inequalities holds when $5 \leq k \leq 26$, and that
\begin{equation}\label{lem69eq1}
-2 < u_2(k) < \frac{1-\sqrt{17}}{2} < u_3(k) < -\frac{4}{3} < u_4(k)
\end{equation}
holds for $k=27$. Next, notice that
\[
\lim_{k \to \infty} u_1(k) = -2, \quad \lim_{k \to \infty} u_2(k) = \frac{1-\sqrt{17}}{2}, \quad \lim_{k \to \infty} u_3(k) = -\frac{4}{3},
\]
and that $u_i(k)$ is an increasing function of $k$ for all $i \in [4]$.
\ignore{ One can verify this by consider the derivative $\frac{d}{dk} u_i(k)$ for $i \in [3]$. For $i=4$, suppose $f(k,1) = \ell$. Thus, there exists $(a,b) \in \LS_+(H_k)$ where
\[
\ell = \frac{\frac{k-1}{k} - b}{\frac{1}{k} - a} = \frac{k-1-kb}{1-ka}.
\]
Now, notice that if we define $\bar{x} \in \mR^{3(k+1)}$ where $\bar{x} = \begin{bmatrix} w_k(a,b) \\ 0 \\ 1 \\ 0 \end{bmatrix}$ belongs to $\LS_+(H_{k+1})$. Thus, by symmetry, $x_{k+1}(\frac{ka}{k+1}, \frac{kb+1}{k+1}) \in \LS_+(H_{k+1})$, and we obtain that
\[
f(k+1, 1) \geq \frac{\frac{k}{k+1} - \frac{kb+1}{k+1}}{\frac{1}{k+1} - \frac{ka}{k+1}} = \frac{k-kb-1}{1-ka} = \ell = f(k,1).
\]}
Thus,~\eqref{lem69eq1} in fact holds for all $k \geq 27$, and our claim follows.
\end{proof}

Now we are ready to prove the following key lemma which bounds the difference between $f(k-1,p-1)$ and $f(k,p)$.

\begin{lemma}\label{lem68}
Given $k \geq 5$ and $\ell \in (u_1(k), u_3(k))$, let 
\[
\gamma \ce (k-2)(9k-10)\ell^2 + 8(k-1)(3k-4)\ell + 16(k-1)^2,
\]
and
\[
h(k,\ell) \ce \frac{4(k-2) \ell + 8(k-1)}{ \sqrt{\gamma} + 3(k-2)\ell + 8(k-1)} - 2 - \ell.
\]
If $f(k-1, p-1) \leq \ell+ h(k,\ell)$, then $f(k,p) \leq \ell$.
\end{lemma}

\begin{proof}
Given $\varepsilon > 0$, define $a \ce \frac{1}{k} + \varepsilon$ and $b \ce \frac{k-1}{k} + \ell \varepsilon$. We solve for $c,d$ so that they satisfy condition (ii) in Lemma~\ref{lem67} and (S5) in Lemma~\ref{lem63} with equality. That is,
\begin{align}
\label{lem68eq2} d-2b-2c &= 1, \\
\label{lem68eq1} (2a + (k-2)c - 2k a^2)(2b + 2(k-1)d - 2kb^2) - (2(k-1)(a-c) - 2kab)^2 &= 0.
\end{align}
To do so, we substitute $d = 2b+2c-1$ into~\eqref{lem68eq1}, and obtain the quadratic equation 
\[
p_2c^2 + p_1c + p_0 =0
\]
where
\begin{align*}
p_2 \ce {}& (k-2)(4(k-1)) - (-2(k-1))^2,\\
p_1 \ce {}&(k-2)(2b + 2(k-1)(2b-1) -2kb^2) + (2a - 2ka^2)(4(k-1)) \\
& -2(-2(k-1))(2(k-1)a - 2kab),\\
p_0 \ce {}& (2a - 2ka^2)(2b + 2(k-1)(2b-1) -2kb^2) - (2(k-1)a - 2kab)^2.
\end{align*}
We then define $c \ce \frac{-p_1 + \sqrt{p_1^2 - 4p_0p_2}}{2p_2}$ (this would be the smaller of the two solutions, as $p_2 < 0$), and $d \ce 2b+2c-1$. First, we assure that $c$ is well defined. If we consider the discriminant $\Delta p \ce p_1^2 - 4p_0p_2$ as a function of $\varepsilon$, then $\Delta p(0) = 0$, and that $\frac{d^2}{d \varepsilon^2} \Delta p(0) > 0$ for all $\ell \in ( u_1(k), u_3(k))$. Thus, there must exist $\varepsilon > 0$ where $\Delta p \geq 0$, and so $c,d$ are well defined.

Next, we verify that $W_k(a,b,c,d) \succeq 0$ for some $\varepsilon > 0$ by checking the conditions from Lemma~\ref{lem63}. First, by the choice of $c,d$, (S5) must hold. Next, define the quantities
\begin{align*}
\theta_1 & \ce c, &
\theta_2 & \ce a-c,\\
\theta_3 & \ce b-d-a+c, &
\theta_4 & \ce 2a + (k-2)c - 2k a^2.
\end{align*}
Notice that at $\varepsilon = 0$, $\theta_i = 0$ for all $i \in [4]$. Next, given a quantity $q$ that depends on $\varepsilon$, we use the notation $q'(0)$ denote the one-sided derivative $\lim_{\varepsilon \to 0^+} \frac{q}{\varepsilon}$. Then it suffices to show that $\theta_i'(0) \geq 0$ for all $i \in [4]$. Observe that
\begin{align*}
\theta_1'(0) \geq 0 &\iff c'(0) \geq 0, &
\theta_2'(0) \geq 0 &\iff c'(0) \leq 1,\\
\theta_3'(0) \geq 0 &\iff c'(0) \leq -1 - \ell , &
\theta_4'(0) \geq 0 &\iff c'(0) \geq \frac{2}{k-2}.
\end{align*}
Now one can check that
\[
c'(0) = \frac{-3k \ell-\sqrt{\gamma} -4k+2\ell+4}{4k-4}.
\]
As a function of $\ell$, $c'(0)$ is increasing over $(u_1(k), u_3(k))$, with
\begin{align*}
\left. c'(0)\right|_{\ell = u_1(k)} &= \frac{2}{k-2}, &
\left. c'(0)\right|_{\ell = u_3(k)} &= \frac{(6k-4)\sqrt{k-1}+10k-12}{(k-2)(9k-10)}.
\end{align*}
Thus, for all $k \geq 5$, we see that $\frac{2}{k-2} \leq c'(0) \leq \min\set{1, -1-\ell}$ for all $\ell \in (u_1(k), u_3(k))$, and so there exists $\varepsilon > 0$ where $W_k(a,b,c,d) \succeq 0$.

Next, for convenience, let 
\begin{align*}
s_1 & \ce s_{k-1}\left(\frac{a-c}{b}, \frac{d}{b} \right), &
s_2 & \ce s_{k-1}\left(\frac{a-c}{1-a-c}, \frac{b-a+c}{1-a-c} \right).
\end{align*}
Notice that both $s_1, s_2$ are undefined at $\varepsilon = 0$, as $\left(\frac{a-c}{b}, \frac{d}{b} \right) = \left(\frac{a-c}{1-a-c}, \frac{b-a+c}{1-a-c} \right) = \left( \frac{1}{k-1}, \frac{k-2}{k-1} \right)$ in this case. Now one can check that
\begin{align*}
\lim_{\varepsilon \to 0^+} s_1 &= \frac{-2 \sqrt{\gamma}-2(k-2)\ell - 8(k-1)}{\sqrt{\gamma} +3(k-2)\ell + 8(k-1)},\\
\lim_{\varepsilon \to 0^+} s_2 &= \frac{(-2k+3) \sqrt{\gamma}-(2k-1)(k-2)\ell - 8(k-1)^2}{(k-2)\sqrt{\gamma} +(3k-2)(k-2)\ell + 8(k-1)^2}.
\end{align*}
Observe that for $k \geq 5$ and for all $\ell \in (u_1(k), 0)$, we have
\begin{align*}
0 &> \frac{-2 \sqrt{\gamma}}{ \sqrt{\gamma}} > \frac{(-2k+3) \sqrt{\gamma}}{(k-2) \sqrt{\gamma}},
\\
0 &> \frac{-2(k-2)\ell - 8(k-1)}{3(k-2)\ell + 8(k-1)} > \frac{-(2k-1)(k-2)\ell - 8(k-1)^2}{(3k-2)(k-2)\ell + 8(k-1)^2}.
\end{align*}
Thus, we conclude that for all $k, \ell$ under our consideration, $s_1 \geq s_2$ for arbitrarily small $\varepsilon > 0$.

Now, notice that $h(k,\ell) = \lim_{\varepsilon \to 0^+} s_1 - \ell$. Thus, if $\ell \in (u_1(k), u_3(k))$, then there exists $\varepsilon > 0$ where the matrix $W_k(a,b,c,d)$ as constructed is positive semidefinite, satisfies $d \geq 2b+2c-1$ (by the choice of $c,d$), with $s_2 \leq s_1 \leq h(k,\ell) + \ell$. Hence, if $f(k-1, p-1) \geq \ell + h(k,\ell)$, then Lemma~\ref{lem67} applies, and we obtain that $f(k,p) \geq \ell$.
\end{proof}

Applying Lemma~\ref{lem68} iteratively, we obtain the following.

\begin{lemma}\label{lem68b}
Given $k \geq 5$, suppose there exists $\ell_1, \ldots, \ell_p \in \mR$ where
\begin{itemize}
\item[(i)]
$\ell_p > u_1(k)$, $\ell_2 < u_3(k-p+2)$, and $\ell_1 < u_4(k-p+1)$;
\item[(ii)]
$\ell_i + h(k-p+i, \ell_i) \leq \ell_{i-1}$ for all $i \in \set{2, \ldots, p}$.
\end{itemize}
Then $r_+(H_k) \geq p+1$.
\end{lemma}

\begin{proof}
First, notice that $\ell_1 < u_4(k-p+1) \leq f(k-p+1,1)$ by Proposition~\ref{prop64}. Then since $\ell_2 < u_3(k-p+2)$ and $\ell_2 + h(k-p+2, \ell_2) \leq \ell_1$, Lemma~\ref{lem68} implies that $\ell_2 \leq f(k-p+2, 2)$. Iterating this argument results in $\ell_i \leq f(k-p+i, i)$ for every $i \in [p]$. In particular, we have $\ell_p \leq f(k,p)$. Since $\ell_p > u_1(k)$, it follows that $r_+(H_k) > p$, and the claim follows.
\end{proof}

Lemmas~\ref{lem68} and~\ref{lem68b} provide a simple procedure of establishing $\LS_+$-rank lower bounds for $H_k$. 

\begin{example}\label{egHkLB}
Let $k = 7$. Then $\ell_2 = -2.39$ and $\ell_1 = \ell_2 + h(7,\ell_2)$ certify that $r_+(H_7) \geq 3$. Similarly, for $k = 10$, one can let $\ell_3 = -2.24, \ell_2 = \ell_3+h(10,\ell_3)$, and $\ell_1 = \ell_2 + h(9,\ell_2)$ and use Lemma~\ref{lem68b} to verify that $r_+(H_{10}) \geq 4$.
\end{example}

Next, we prove a lemma that will help us obtain a lower bound for $r_+(H_k)$ analytically.

\begin{lemma}\label{lem610}
For all $k \geq 5$ and $\ell \in (u_1(k), u_2(k))$, $h(k,\ell) \leq \frac{2}{k-2}$.
\end{lemma}

\begin{proof}
One can check that the equation $h(k,\ell) = \frac{2}{k-2}$ has three solutions: $\ell = u_1(k), u_2(k)$, and $ \frac{k-4- \sqrt{17k^2-48k+32}}{2(k-2)}$ (which is greater than $u_2(k)$). Also, notice that $\left. \frac{\partial }{\partial \ell} h(k,\ell) \right|_{\ell = u_1(k)} = -\frac{1}{k-1} < 0$. Since $h(k,\ell)$ is a continuous function of $\ell$ over $(u_1(k), u_2(k))$, it follows that $h(k,\ell) \leq \frac{2}{k-2}$ for all $\ell$ in this range.
\end{proof}

We are finally ready to prove the main result of this section.

\begin{theorem}\label{thm611}
The $\LS_+$-rank of $H_k$ is
\begin{itemize}
\item
at least $2$ for $4 \leq k \leq 6$;
\item
at least $3$ for $7 \leq k \leq 9$;
\item
at least $\lfloor 0.19(k-2)\rfloor + 3$ for all $k \geq 10$.
\end{itemize}
\end{theorem}

\begin{proof}
First, $r_+(H_4) \geq 2$ follows from Corollary~\ref{corH_4}, and $r_+(H_7) \geq 3$ was shown in Example~\ref{egH7} and again in Example~\ref{egHkLB}. Moreover, one can use the approach illustrated in Example~\ref{egHkLB} to verify that $r_+(H_k) \geq \lfloor 0.19(k-2)\rfloor + 3$ for all $k$ where $10 \leq k \leq 49$. Thus, we shall assume that $k \geq 50$ for the remainder of the proof.

Let $q \ce \lfloor 0.19(k-2)\rfloor$, let $\varepsilon >0$ that we set to be sufficiently small, and define
\[
\ell_i \ce \varepsilon + u_1(k) + \sum_{j=1}^{q+2-i} \frac{2}{k-1-j}.
\]
for every $i \in [q+2]$. (We aim to subsequently apply Lemma~\ref{lem68b} with $p = q+2$.) Now notice that 
\[
\sum_{j=1}^{q} \frac{2}{k-1-i} \leq \int_{k-2-q}^{k-2} \frac{2}{t}~dt = 2\ln\left( \frac{k-2}{k-2-q} \right),
\]
Also, notice that
\[
u_2(k-q) - u_1(k) \geq u_2\left(\frac{4}{5}k\right) - u_1(k),
\]
as $u_2$ is an increasing function in $k$ and $q \leq \frac{k}{5}$. Also, one can check that $\overline{w}(k) \ce u_2\left(\frac{4}{5}k\right) - u_1(k)$ is also an increasing function for all $k \geq 5$. Next, we see that
\[
2\ln\left( \frac{k-2}{k-2-q} \right) \leq \overline{w}(50) \iff q \leq \left( 1 - \frac{1}{\exp(\overline{w}(50) /2)} \right) (k-2)
\]
Since $ 1 - \frac{1}{\exp(\overline{w}(50) /2)} > 0.19$, the first inequality does hold by the choice of $q$. Hence,
\[
\ell_2 - \varepsilon = u_1(k) + \sum_{j=1}^{q} \frac{2}{k-1-j} < u_2(k-q).
\]
Thus, we can choose $\varepsilon$ sufficiently small so that $\ell_2 < u_2(k-q)$. Then Lemma~\ref{lem610} implies that $\ell_i + h(k-q-2+i, \ell_i) \leq \ell_{i-1}$ for all $i \in \set{2, \ldots, q+2}$. Also, for all $k \geq 50$, $u_2(k-q) + \frac{1}{k-q-1} < u_4(k-q-1)$. Thus, we obtain that $\ell_1 < u_4(k-q-1)$, and it follows from Lemma~\ref{lem68b} that $r_+(H_k) \geq q+3$.
\end{proof}

Since $H_k$ has $3k$ vertices, Theorem~\ref{thm611} (and the fact that $r_+(H_3) =1$) readily implies Theorem~\ref{thmHk}. In other words, we now know that for every $\ell \in \mN$, there exists a graph on no more than $16\ell$ vertices that has $\LS_+$-rank $\ell$. 

\section{Chv{\'a}tal--Gomory rank of $\STAB(H_k)$}\label{sec5}

In this section we determine the degree of hardness of $\STAB(H_k)$ relative to another well-studied cutting plane procedure that is due to Chv{\'a}tal~\cite{Chvatal73} with earlier ideas from Gomory~\cite{Gomory58}. Given a set $P \subseteq [0,1]^n$, if $a^{\top}x \leq \b$ is a valid inequality of $P$ and $a \in \mZ^n$, we say that $a^{\top}x \leq \lfloor \b \rfloor$ is a \emph{Chv{\'a}tal--Gomory cut} for $P$. Then we define $\CG(P)$, the \emph{Chv{\'a}tal--Gomory closure} of $P$, to be the set of points that satisfy all Chv{\'a}tal--Gomory cuts for $P$. Note that $\CG(P)$ is a closed convex set which contains all integral points in $P$.
Furthermore, given an integer $p \geq 2$, we can recursively define $\CG^p(P) \ce \CG( \CG^{p-1}(P))$. Then given any valid linear inequality of $P_I$, we can define its \emph{$\CG$-rank} (relative to $P$) to be the smallest integer $p$ for which the linear inequality is valid for $\CG^p(P)$. 

In Section~\ref{sec4}, we proved that the inequality~\eqref{lem62eq0} has $\LS_+$-rank $\Theta(|V(H_k)|)$. This implies that the inequality~\eqref{lem61eq0} also has $\LS_+$-rank $\Theta(|V(H_k)|)$ (since it was shown in the proof of Lemma~\ref{lem61}(ii) that~\eqref{lem62eq0} is a non-negative linear combination of~\eqref{lem61eq0}). Here, we show that~\eqref{lem61eq0} has $\CG$-rank $\Theta(\log(|V(H_k)|))$.

\begin{theorem}\label{thmCGHk}
Let $d$ be the $\CG$-rank of the facet~\eqref{lem61eq0} of $\STAB(H_k)$ relative to $\FRAC(H_k)$. Then
\[
\log_4 \left( \frac{3k-7}{2} \right) < d \leq \log_2\left(k-1\right).
\]
\end{theorem}

Before providing a proof of Theorem~\ref{thmCGHk}, first, we need a lemma about the valid inequalities of $\STAB(H_k)$.

\begin{lemma}\label{lemCG1}
Suppose $a^{\top}x \leq \b$ is valid for $\STAB(H_k)$ where $a \in \mZ_+^{V(H_k)} \setminus \set{0}$. Then $\frac{\b}{a^{\top}\bar{e}} > \frac{1}{3}$.
\end{lemma}

\begin{proof}
We consider two cases. First, suppose that $a_{j_1} = 0$ for all $j \in [k]$. Since $[k]_p$ is a stable set in $H_k$ for $p \in \set{0,1,2}$, observe that
\[
a^{\top}\bar{e} = a^{\top}\left(\chi_{[k]_0} + \chi_{[k]_1} + \chi_{[k]_2}\right) \leq \b + 0 + \b = 2\b.
\]
Thus, we obtain that $\frac{\b}{a^{\top}\bar{e}} \geq \frac{1}{2} > \frac{1}{3}$ in this case. Otherwise, we may choose $j \in [k]$ where $a_{j_1} > 0$. Consider the stable sets 
\[
S_0 \ce ([k]_0 \setminus \set{j_0}) \cup \set{j_1},
S_1 \ce ([k]_1 \setminus \set{j_1}) \cup \set{j_0, j_2}, 
S_2 \ce ([k]_2 \setminus \set{j_2}) \cup \set{j_1}.
\]
Now $\chi_{S_0} + \chi_{S_1} + \chi_{S_2} = \bar{e} + e_{j_1}$. Since $a_{j_1} > 0$, this implies that 
\[
a^{\top}\bar{e} < a^{\top} (\bar{e} + e_{j_1}) = a^{\top} \left( \chi_{S_0} + \chi_{S_1} + \chi_{S_2} \right) \leq 3\b,
\]
 and so $\frac{\b}{a^{\top}\bar{e}} > \frac{1}{3}$ in this case as well.
\end{proof}

We will also need the following result.

\begin{lemma}\label{lemCG2}
\cite[Lemma 2.1]{ChvatalCH89} Let $P \subseteq \mR^n$ be a rational polyhedron. Given $u, v \in \mR^n$ and positive real numbers $m_1, \ldots, m_d \in \mR$, define
\[
x^{(i)} \ce u - \left( \sum_{i=1}^d \frac{1}{m_i} \right) v
\]
for all $i \in [d]$. Suppose
\begin{itemize}
\item[(i)]
$u \in P$, and
\item[(ii)]
for all $i \in [d]$, $a^{\top} x^{(i)} \leq \b$, for every inequality $a^{\top}x \leq \b$ that is valid for $P_I$ and satisfies $a \in \mZ^n$ and $a^{\top}v < m_i$.
\end{itemize}
Then $x^{(i)} \in \CG^i(P)$ for all $i \in [d]$.
\end{lemma}

We are now ready to prove Theorem~\ref{thmCGHk}.

\begin{proof}[Proof of Theorem~\ref{thmCGHk}]
We first prove the rank lower bound. Given $d \geq 0$, let $k \ce \frac{1}{3} (2^{2d+1} + 7)$ (then $d = \log_4\left( \frac{3k-7}{2}\right)$). We show that the $\CG$-rank of the inequality $\sum_{i \in B_{j,j'}} x_{i} \leq k-1$ is at least $d+1$ using Lemma~\ref{lemCG2}.
 
Let $u \ce \frac{1}{2}\bar{e}, v \ce \bar{e}$, and $m_i \ce 2^{2i+1}$ for all $i \in [d]$. Then notice that $x^{(i)} = \frac{2^{2i+1}+ 1}{3 \cdot 2^{2i+1}} \bar{e}$ for all $i \in [d]$. Now suppose $a^{\top}x \leq \b$ is valid for $\STAB(H_k)$ where $a$ is an integral vector and $a^{\top}v < m_i$ (which translates to $a^{\top}\bar{e} < 2^{2i+1}$). Now Lemma~\ref{lemCG1} implies that $\frac{\b}{a^{\top}\bar{e}} > \frac{1}{3}$. Furthermore, using the fact that $\b, a^{\top}\bar{e}$ are both integers, $a^{\top}\bar{e} < 2^{2i+1}$, and $2^{2i+1} \equiv 2 ~(\tn{mod}~3)$, we obtain that $\frac{\b}{a^{\top}\bar{e}} \geq \frac{2^{2i+1} +1}{ 3 \cdot 2^{2i+1}}$, which implies that $a^{\top} x^{(i)} \leq \b$. Thus, it follows from Lemma~\ref{lemCG2} that $x^{(i)} \in \CG^i(H_k)$ for every $i \in [d]$.

In particular, we obtain that $x^{(d)} = \frac{2^{2d+1}+1}{3 \cdot 2^{2d+1}} \bar{e} \in \CG^d(H_k)$. However, notice that $x^{(d)}$ violates the inequality $\sum_{i \in B_{j,j'}} x_{i} \leq k-1$ for $\STAB(H_k)$, as 
\[
\frac{k-1}{ |B_{j,j'}|} = \frac{k-1}{3k-4} = \frac{2^{2d+1}+4}{3 \cdot 2^{2d+1}+3} > \frac{2^{2d+1}+1}{3 \cdot 2^{2d+1}}.
\]
Next, we turn to proving the rank upper bound. Given $d \in \mN$, let $k \ce 2^d+1$ (then $d = \log_2(k-1)$). We prove that $\sum_{i \in B_{j,j'}} x_{i} \leq k-1$ is valid for $\CG^d(H_k)$ by induction on $d$. When $d=1$, we see that $k=3$ and $B_{j,j'}$ induces a $5$-cycle, so the claim holds.

Now assume $d \geq 2$, and $k = 2^d+1$. Let $j,j'$ be distinct, fixed indices in $[k]$. By the inductive hypothesis, if we let $T \subseteq [k] \setminus \set{j,j'}$ where $|T| = 2^{d-1}-1$, then the inequality
\begin{equation}\label{thmCGHkeq1}
x_{j_0} + x_{j_2'} + \sum_{ \ell \in T} \left( x_{\ell_0} + x_{\ell_1} + x_{\ell_2} \right) \leq 2^{d -1}
\end{equation}
is valid for $\CG^{d-1}(H_k)$ (since the subgraph induced by $\set{ \ell_0, \ell_1, \ell_2 : \ell \in T} \cup \set{ j_0, j_2'}$ is a copy of that by $B_{j,j'}$ in $H_{k-1}$). Averaging the above inequality over all possible choices of $T$, we obtain that
\begin{equation}\label{propCGUBeq1}
x_{j_0} + x_{j_2'} + \frac{2^{d-1} -1}{k-2} \sum_{ \ell \in [k] \setminus \set{j,j'}} \left( x_{\ell_0} + x_{\ell_1} + x_{\ell_2} \right) \leq 2^{d-1} 
\end{equation}
is valid for $\CG^{d-1}(H_k)$. Next, using~\eqref{thmCGHkeq1} plus two edge inequalities, we obtain that for all $T \subseteq [k] \setminus \set{j,j'}$ where $|T| = 2^{d-1}+1$, the inequality
\[
\sum_{ \ell \in T} \left( x_{\ell_0} + x_{\ell_1} + x_{\ell_2} \right) \leq 2^{d -1} + 2
\]
is valid for $\CG^{d-1}(H_k)$. Averaging the above inequality over all choices of $T$, we obtain
\begin{equation}\label{propCGUBeq2}
\frac{2^{d-1} +1}{k-2} \sum_{ \ell \in [k] \setminus \set{j,j'} } \left( x_{\ell_0} + x_{\ell_1} + x_{\ell_2} \right) \leq 2^{d -1} + 2.
\end{equation}
Taking the sum of~\eqref{propCGUBeq1} and $\frac{k - 2^{d-1}-1}{2^{d-1}+1}$ times~\eqref{propCGUBeq2}, we obtain that
\begin{equation}\label{propCGUBeq3}
x_{j_0} + x_{j_2'} + \sum_{ \ell \in [k] \setminus \set{j,j'}} \left( x_{\ell_0} + x_{\ell_1} + x_{\ell_2} \right) \leq \frac{k-2^{d-1}-1}{2^{d-1}+1} (2^{d -1} + 2) + 2^{d-1}
\end{equation}
is valid for $\CG^{d-1}(H_k)$. Now observe that the left hand side of~\eqref{propCGUBeq3} is simply $\sum_{i \in B_{j,j'}} x_{i}$. On the other hand, the right hand side simplifies to $k - 2 + \frac{k}{2^{d-1}+1}$. Since $k = 2^d+1, 1 < \frac{k}{2^{d-1}+1} < 2$, and so the floor of the right hand side of~\eqref{propCGUBeq3} is $k-1$. This shows that the inequality $\sum_{i \in B_{j,j'}} x_{i} \leq k-1$ has $\CG$-rank at most $d$.
\end{proof}

Thus, we conclude that the facet~\eqref{lem61eq0} has $\LS_+$-rank $\Theta(|V(H_k)|)$ and  $\CG$-rank $\Theta(\log(|V(H_k)|))$. We remark that the two results are incomparable in terms of computational complexity since it is generally $\mathcal{NP}$-hard to optimize over $\CG^p(P)$ even for $p = O(1)$. These rank bounds for $H_k$ also provides an interesting contrast with the aforementioned example involving line graphs of odd cliques from~\cite{StephenT99}, which have $\LS_+$-rank $\Theta(\sqrt{|V(G)|})$ and $\CG$-rank $\Theta(\log(|V(G)|))$. 	In the context of the matching problem, odd cliques have $\CG$-rank one with respect to the fractional matching polytope. This last claim follows from an observation of
Chv\'{a}tal~\cite[pp. 309-310]{Chvatal73}.

\section{Symmetric graphs with high $\LS_+$-ranks}\label{sec6}

So far we have established that there exists a family of graphs (e.g., $\left\{H_k \, \, : \,\, k \geq 2\right\}$) which have $\LS_+$-rank $\Theta(|V(G)|)$. However, the previous best result in this context $\Theta(\sqrt{|V(G)|})$ was achieved by a vertex-transitive family of graphs (line graphs of odd cliques). In this section, we show that there exists a family of vertex-transitive graphs which also have $\LS_+$-rank $\Theta(|V(G)|)$.

\subsection{The $L_k$ construction}\label{sec6.1}

In this section, we look into a procedure that is capable of constructing highly symmetric graphs with 
high $\LS_+$-rank by virtue of containing $H_k$ as an induced subgraph. 

\begin{definition}\label{defnLk1}
Given a graph $G$ and an integer $k \geq 2$, define the graph $L_k(G)$ such that $V(L_k(G)) \ce \set{ i_p : i \in [k], p \in V(G)}$, and vertices $i_p, j_q$ are adjacent in $L_k(G)$ if
\begin{itemize}
\item
$i=j$ and $\set{p,q} \in E(G)$, or
\item
$i \neq j$, $p \neq q$, and $\set{p,q} \not\in E(G)$.
\end{itemize}
\end{definition}

For example, let $C_4$ be the $4$-cycle with $V(C_4) \ce \set{0,1,2,3}$ and 
\[
E(C_4) \ce \set{ \set{0,1}, \set{1,2}, \set{2,3}, \set{3,0}}.
\]
Figure~\ref{figL_k} illustrates the graphs $L_2(C_4)$ and $L_3(C_4)$.

\def\x{360/4}
\def\y{0.30}
\def\sc{2}

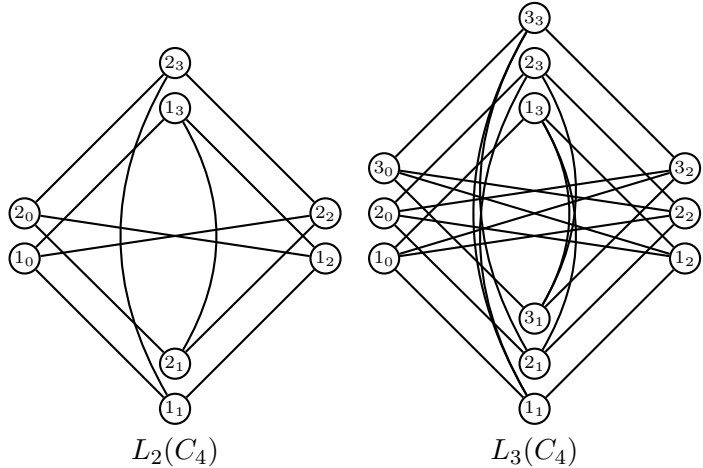
\begin{figure}[ht!]
\begin{center}
\begin{tabular}{cc}

\begin{tikzpicture}
[scale=\sc, thick,main node/.style={circle, minimum size=4mm, inner sep=0.1mm,draw,font=\tiny\sffamily}]

\node[main node] at ({cos((-1)*\x)},{sin((-1)*\x)}) (1) {$1_1$};
\node[main node] at ({cos((0)*\x)},{sin((0)*\x)}) (2) {$1_2$};
\node[main node] at ({cos(1*\x)},{sin((1)*\x)}) (3) {$1_3$};
\node[main node] at ({cos(2*\x)},{sin((2)*\x)}) (4) {$1_0$};

\node[main node] at ({cos((-1)*\x)},{sin((-1)*\x) + \y}) (5) {$2_1$};
\node[main node] at ({cos((0)*\x)},{sin((0)*\x)+ \y}) (6) {$2_2$};
\node[main node] at ({cos(1*\x)},{sin((1)*\x)+ \y}) (7) {$2_3$};
\node[main node] at ({cos(2*\x)},{sin((2)*\x)+ \y}) (8) {$2_0$};

 \path[every node/.style={font=\sffamily}]
(1) edge (2)
(2) edge (3)
(3) edge (4)
(4) edge (1)
(5) edge (6)
(6) edge (7)
(7) edge (8)
(8) edge (5)
(1) edge[bend left] (7)
(3) edge[bend left] (5)
(2) edge (8)
(4) edge (6);
\end{tikzpicture}
&

\begin{tikzpicture}
[scale=\sc, thick,main node/.style={circle, minimum size=4mm, inner sep=0.1mm,draw,font=\tiny\sffamily}]

\node[main node] at ({cos((-1)*\x)},{sin((-1)*\x)}) (1) {$1_1$};
\node[main node] at ({cos((0)*\x)},{sin((0)*\x)}) (2) {$1_2$};
\node[main node] at ({cos(1*\x)},{sin((1)*\x)}) (3) {$1_3$};
\node[main node] at ({cos(2*\x)},{sin((2)*\x)}) (4) {$1_0$};

\node[main node] at ({cos((-1)*\x)},{sin((-1)*\x) + \y}) (5) {$2_1$};
\node[main node] at ({cos((0)*\x)},{sin((0)*\x)+ \y}) (6) {$2_2$};
\node[main node] at ({cos(1*\x)},{sin((1)*\x)+ \y}) (7) {$2_3$};
\node[main node] at ({cos(2*\x)},{sin((2)*\x)+ \y}) (8) {$2_0$};

\node[main node] at ({cos((-1)*\x)},{sin((-1)*\x) + 2*\y}) (9) {$3_1$};
\node[main node] at ({cos((0)*\x)},{sin((0)*\x)+ 2*\y}) (10) {$3_2$};
\node[main node] at ({cos(1*\x)},{sin((1)*\x)+ 2*\y}) (11) {$3_3$};
\node[main node] at ({cos(2*\x)},{sin((2)*\x)+ 2*\y}) (12) {$3_0$};

 \path[every node/.style={font=\sffamily}]
(1) edge (2)
(2) edge (3)
(3) edge (4)
(4) edge (1)
(5) edge (6)
(6) edge (7)
(7) edge (8)
(8) edge (5)
(9) edge (10)
(10) edge (11)
(11) edge (12)
(12) edge (9)
(1) edge[bend left] (7)
(1) edge[bend left] (11)
(3) edge[bend left] (5)
(3) edge[bend left] (9)
(2) edge (8)
(2) edge (12)
(4) edge (6)
(4) edge (10)
(5) edge[bend left] (11)
(6) edge (12)
(7) edge[bend left] (9)
(8) edge (10);
\end{tikzpicture}
\\
$L_2(C_4)$ & $L_3(C_4)$ 
\end{tabular}
\end{center}
\caption{Illustrating the $L_k(G)$ construction on the $4$-cycle}\label{figL_k}
\end{figure}

Moreover, notice that if we define $P_2$ to be the graph which is a path of length $2$, with $V(P_2) \ce \set{0,1,2}$ and $E(P_2) \ce \set{ \set{0,1}, \set{1,2}}$, then $L_k(P_2) = H_k$ for every $k \geq 2$. Thus, we obtain the following.

\begin{proposition}\label{propLk}
Let $G$ be a graph that contains $P_2$ as an induced subgraph. Then the $\LS_+$-rank lower bound in Theorem~\ref{thm611} for $H_k$ also applies for $L_k(G)$.
\end{proposition}

\begin{proof}
Since $G$ contains $P_2$ as an induced subgraph, there must exist vertices $a,b,c \in V(G)$ where $\set{a,b}, \set{b,c} \in E(G)$, and $\set{a,c} \not\in E(G)$. Then the subgraph of $L_k(G)$ induced by the vertices in $\set{ i_p : i \in [k], p \in \set{a,b,c}}$ is exactly $L_k(P_2) = H_k$. Thus, it follows from Lemma~\ref{lemInducedSubgraph} that $r_+(L_k(G)) \geq r_+(H_k)$.
\end{proof}

Since $L_k(C_4)$ has $4k$ vertices, Theorem~\ref{thm611} and Proposition~\ref{propLk} immediately imply the following.

\begin{theorem}\label{thmLkC4}
Let $k \geq 3$ and $G \ce L_k(C_4)$. Then $r_+(G) \geq \frac{1}{22} |V(G)|$.
\end{theorem}

Since $\set{ L_k(C_4) : k \geq 3}$ is a family of vertex-transitive graphs, Theorem~\ref{thmLkC4} can also be proved directly by utilizing versions of the techniques in Sections~\ref{sec3} and~\ref{sec4}. The graphs $L_k(C_4)$ are particularly noteworthy because $C_4$ is the smallest vertex-transitive graph that contains $P_2$ as an induced subgraph. In general, observe that if $G$ is vertex-transitive, then so is $L_k(G)$. Thus, we now know that there exists a family of vertex-transitive graphs $G$ with $r_+(G) = \Theta(|V(G)|)$. 

\subsection{Generalizing the $L_k$ construction}\label{sec6.2}

Next, we study one possible generalization of the aforementioned $L_k$ construction, and mention some interesting graphs it produces.

\begin{definition}\label{defnLk2}
 Given graphs $G_1, G_2$ on the same vertex set $V$, and an integer $k \geq 2$, define $L_k(G_1, G_2)$ to be the graph with vertex set $\set{ i_p : i \in [k], p \in V}$. Vertices $i_p, j_q$ are adjacent in $L_k(G_1,G_2)$ if
\begin{itemize}
\item
$i=j$ and $\set{p,q} \in E(G_1)$, or
\item
$i \neq j$ and $\set{p,q} \in E(G_2)$.
\end{itemize}
\end{definition}

Thus, when $G_2 = \overline{G_1}$ (the complement of $G_1$), then $L_k(G_1, G_2)$ specializes to $L_k(G_1)$. Next, given $\ell \in \mN$ and $S \subseteq [\ell]$, let $Q_{\ell, S}$ denote the graph whose vertices are the $2^{\ell}$ binary strings of length $\ell$, and two strings are joined by an edge if the number of positions they differ by is contained in $S$. For example, $Q_{\ell, \set{1}}$ gives the $\ell$-cube. Then we have the following.

\begin{proposition}\label{propLkQl}
For every $\ell \geq 2$, 
\[
L_4( Q_{\ell, \set{1}}, Q_{\ell, \set{\ell}}) = Q_{\ell+2, \set{1,\ell+2}}.
\]
\end{proposition}

\begin{proof}
Let $G \ce L_4( Q_{\ell, \set{1}}, Q_{\ell, \set{\ell}})$. Given $i_p \in V(G)$ (where $i \in [4]$ and $p \in \set{0,1}^{\ell})$, we define the function
\[
f(i_p) \ce 
\begin{cases}
00p & \tn{if $i=1$;}\\
01\overline{p} & \tn{if $i=2$;}\\
10\overline{p} & \tn{if $i=3$;}\\
11p & \tn{if $i=4$.}\\
\end{cases}
\]
Note that $\overline{p}$ denotes the binary string obtained from $p$ by flipping all $\ell$ bits. Now we see that $\set{i_p, j_q} \in E(G)$ if and only if $f(i_p)$ and $f(j_q)$ differ by either $1$ bit or all $\ell+2$ bits, and the claim follows.
\end{proof}

The graph $Q_{k, \set{1,k}}$ is known as the folded-cube graph, and Proposition~\ref{propLkQl} implies the following.

\begin{corollary}\label{corFoldedCubes}
Let $G \ce Q_{k, \set{1,k}}$ where $k \geq 3$. Then $r_+(G) \geq 2$ if $k$ is even, and $r_+(G) = 0$ if $k$ is odd.
\end{corollary}

\begin{proof}
First, observe that when $k$ is odd, $Q_{k, \set{1,k}}$ is bipartite, and hence has $\LS_+$-rank $0$. Next, assume $k \geq 4$ is even. Notice that $Q_{k-2,\set{1}}$ contains a path of length $k-2$ from the all-zeros vertex to the all-ones vertex, while $Q_{k-2, \set{k-2}}$ joins those two vertices by an edge. Thus, $Q_{k, \set{1,k}} = L_4(Q_{k-2,\set{1}}, Q_{k-2, \set{k}})$ contains the induced subgraph $L_4(P_{k-2})$ (where $P_{k-2}$ denotes the graph that is a path of length $k-2$). Since $k-2$ is even, we see that $L_4(P_{k-2})$ can be obtained from $L_4(P_2) = H_4$ by odd subdivision of edges (i.e., replacing edges by paths of odd lengths). Thus, it follows from~\cite[Theorem 16]{LiptakT03} that $r_+(L_4(P_{k-2})) \geq 2$, and consequently $r_+(Q_{k,\set{1,k}}) \geq 2$.
\end{proof}

\begin{example}
The case $k=4$ in Corollary~\ref{corFoldedCubes} is especially noteworthy. In this case $G \ce Q_{4, \set{1,4}}$ is the ($5$-regular) Clebsch graph. Observe that $G \ominus i$ is isomorphic to the Petersen graph (which has $\LS_+$-rank $1$) for every $i \in V(G)$. Thus, together with Corollary~\ref{corFoldedCubes} we obtain that the Clebsch graph has $\LS_+$-rank $2$.

Alternatively, one can show that $r_+(G) \geq 2$ by using the fact that the second largest eigenvalue of $G$ is $1$. Then it follows from~\cite[Proposition 8]{AuLT23} that $ \max \set{\bar{e}^{\top}x : x \in \LS_+(G)} \geq 6$, which shows that $r_+(G) \geq 2$ since the largest stable set in $G$ has size $5$.
\end{example}

We remark that the Clebsch graph is also special in the following aspect. Given a vertex-transitive graph $G$, we say that $G$ is \emph{transitive under destruction} if $G \ominus i$ is also vertex-transitive for every $i \in V(G)$. As mentioned above, destroying any vertex in the Clebsch graph results in the Petersen graph, and so the Clebsch graph is indeed transitive under destruction. On the other hand, even though $L(G_1, G_2)$ is vertex-transitive whenever $G_1, G_2$ are vertex-transitive, the Clebsch graph is the only example which is transitive under destruction we could find using the $L_k$ construction. For instance, one can check that $Q_{k, \set{1,k}} \ominus i$ is not a regular graph for any $k \geq 5$. Also, observe that the Clebsch graph can indeed be obtained from the ``regular'' $L_k$ construction defined in Definition~\ref{defnLk1}, as
\[
Q_{4, \set{1,4}} = L_4( Q_{2, \set{1}}, Q_{2, \set{2}} ) = L_4( C_4, \overline{C_4}) = L_4(C_4).
\]
However, one can check that $L_k(C_{\ell})$ is transitive under destruction if and only if $(k,\ell) = (4,4)$ (i.e., the Clebsch graph example), and that $L_k(K_{\ell, \ell})$ is transitive under destruction if and only if $(k,\ell) = (4,2)$ (i.e., the Clebsch graph example again). It would be fascinating to see what other interesting graphs can result from the $L_k$ construction.

\section{Some Future Research Directions}\label{sec7}

In this section, we mention some follow-up questions to our work in this manuscript that could lead to interesting future research.

\begin{problem}\label{pro1}
What is the exact $\LS_+$-rank of $H_k$?
\end{problem}

While we showed that $r_+(H_k) \geq 0.19k$ asymptotically in Section~\ref{sec4}, there is likely room for improvement for this bound. First, Lemma~\ref{lem65} is not sharp. In particular, the assumptions needed for $Y(e_0-e_{1_0}), Y(e_0-e_{1_1}) \in \LS_+^{p-1}(H_k)$ are sufficient but not necessary. Using CVX, a package for specifying and solving convex programs~\cite{CVX, GrantB08} with SeDuMi~\cite{Sturm99}, we obtained that $r_+(H_6) \geq 3$. However, there do not exist $a,b,c,d$ that would satisfy the assumptions of Lemma~\ref{lem65} for $k=6$.

Even so, using Lemma~\ref{lem68b} and the approach demonstrated in Example~\ref{egHkLB}, we found computationally that $r_+(H_k) > 0.25k$ for all $3 \leq k \leq 10000$. One reason for the gap between this computational bound and the analytical bound given in Theorem~\ref{thm611} is that the analytical bound only takes advantage of squeezing $\ell_i$'s over the interval $(u_1(k), u_2(k))$. Since we were able to show that $h(k,\ell) = \Theta(\frac{1}{k})$ over this interval (Lemma~\ref{lem610}), this enabled us to establish a $\Theta(k)$ rank lower bound. Computationally, we see that we could get more $\ell_i$'s in over the interval $(u_2(k), u_3(k))$. However, over this interval, $h(k, \ell)$ is an increasing function that goes from $\frac{2}{k-2}$ at $u_2(k)$ to $\Theta\left(\frac{1}{\sqrt{k}}\right)$ at $u_3(k)$. This means that simply bounding $h(k,\ell)$ from above by $h(k, u_3(k))$ would only add an additional factor of $\Theta(\sqrt{k})$ in the rank lower bound. Thus, improving the constant factor in Theorem~\ref{thm611} would seem to require additional insights.

As for an upper bound on $r_+(H_k)$, we know that $r_+(H_4) = 2$, and $r_+(H_{k+1})\leq r_+(H_k) + 1$ for all $k$. This gives the obvious upper bound of $r_+(H_k) \leq k-2$. It would be interesting to obtain sharper bounds or even nail down the exact $\LS_+$-rank of $H_k$.

\begin{problem}\label{pro2}
Given $\ell \in \mN$, is there a graph $G$ with $|V(G)| = 3\ell$ and $r_+(G) = \ell$?
\end{problem}

Theorem~\ref{thmNover3} raises the natural question: Are there graphs on $3\ell$ vertices that have $\LS_+$-rank exactly $\ell$? Such $\ell$-minimal graphs have been found for $\ell=2$~\cite{LiptakT03} and for $\ell=3$~\cite{EscalanteMN06}.
Thus, results from~\cite{LiptakT03, EscalanteMN06} show that the answer is ``yes'' for $\ell = 1,2,3$. In \cite{AuT24}, we
construct the first $4$-minimal graph which shows that the answer is also ``yes'' for $\ell = 4$. Does the pattern continue for larger $\ell$? And more importantly, how can we verify the $\LS_+$-rank of these graphs analytically or algebraically, as opposed to primarily relying on specific numerical certificates?

\begin{problem}
What can we say about the lift-and-project ranks of graphs for other positive semidefinite lift-and-project operators? To start with some concrete questions for this research problem, what are the solutions of Problems~\ref{pro1} and \ref{pro2} when we replace $\LS_+$ with $\Las, \BZ_+, \Theta_p$, or $\SA_+$?
\end{problem}

After $\LS_+$, many stronger semidefinite lift-and-project operators (such as $\Las$~\cite{Lasserre01},\\
$\BZ_+$~\cite{BienstockZ04}, $\Theta_p$~\cite{GouveiaPT10}, and $\SA_+$~\cite{AuT16}) have been proposed. While these stronger operators are capable of producing tighter relaxations than $\LS_+$, these SDP relaxations can also be more computationally challenging to solve. For instance, while the $\LS_+^p$-relaxation of a set $P \subseteq [0,1]^n$ involves $O(n^p)$ PSD constraints of order $O(n)$, the operators $\Las^p, \BZ_+^p$ and $\SA_+^p$ all impose one (or more) PSD constraint of order $\Omega(n^p)$ in their formulations. It would be interesting to determine the corresponding properties of graphs which are minimal with respect to these stronger lift-and-project operators.

\section*{Declarations}

{\bf Conflict of interest:} The authors declare that they have no conflict of interest.

\bibliographystyle{alpha}
\bibliography{ref}

\end{document}